\documentclass[twoside]{amsart}
\date{}
\author{Kurt Foster}
\title{HT90 and ``simplest'' number fields}
\subjclass[2000]{Primary 12E10, 12F12; Secondary 12F05, 11R09, 20B25, 11R11, 11R20, 11R27, 11B39, 11D57} 
\usepackage{amsmath,amsthm,amssymb,amsfonts,url}
\newtheorem{thm}{Theorem}[section]
\newtheorem{cor}[thm]{Corollary}
\newtheorem{lem}[thm]{Lemma}
\newtheorem{prop}[thm]{Proposition}
\theoremstyle{definition}
\newtheorem{defn}{Definition}[section]

\newtheorem*{exmps}{Examples}

\theoremstyle{remark}
\newtheorem*{rem}{Remark}
\newtheorem*{rems}{Remarks}

\numberwithin{equation}{section}

\newcommand{\C}{\mathbb{C}} 
\newcommand{\Q}{\mathbb{Q}} 
\newcommand{\R}{\mathbb{R}} 
\newcommand{\Z}{\mathbb{Z}} 
\newcommand{\kk}{\mathbf{k}} 

\let\abs=\envert

\usepackage[colorlinks=true, pdfstartview=FitV, linkcolor=red, citecolor=blue, urlcolor=blue]{hyperref}

\begin{document}
\maketitle

\begin{abstract} A standard formula \eqref{e:1} leads to a proof of HT90, but requires proving the existence of $\theta$ such that $\alpha\ne 0$, so that $\beta=\alpha/\sigma(\alpha)$.

We instead impose the condition \eqref{e:M}, that taking $\theta=1$ makes $\alpha=0$.  Taking $n=3$, we recover Shanks's simplest cubic fields.  The ``simplest'' number fields of degrees $3$ to $6$, Washington's cyclic quartic fields, and a certain family of totally real cyclic extensions of $\Q(\cos(\pi/4m))$ all have defining polynomials whose zeroes satisfy \eqref{e:M}.

Further investigation of \eqref{e:M} for $n=4$ leads to an elementary algebraic construction of a $2$-parameter family of octic polynomials with ``generic'' Galois group $_{8}T_{11}$.  Imposing an additional algebraic condition on these octics produces a new family of cyclic quartic extensions. This family includes the ``simplest'' quartic fields and Washington's cyclic quartic fields as special cases.

We obtain more detailed results on our octics when the parameters are algebraic integers in a number field.  In particular, we identify certain sets of special units, including exceptional sequences of $3$ units, and give some of their properties.  
\end{abstract}
%
%
\section{Introduction}
\label{s:intro}
Hilbert's Theorem 90 (See, e.g.\ \cite{lang:algebra}) characterizes elements $\beta$ of norm 1 in a cyclic extension $L/\kk$ of degree $n$ with Galois group $G = \langle\sigma\rangle$; one has
\begin{equation}
\tag{HT90} \label{e:ht90}
\mathcal{N}_{L/\kk}(\beta) = 1\; \text{if, and only if,}\; \beta = \alpha/\sigma(\alpha)\,\text{for some}\,\alpha \in L.
\end{equation}
``If'' is obvious.  The usual proof of ``only if '' uses the formula
\begin{equation}
\tag{1}\label{e:1}
\alpha = \theta + \sigma(\theta)\beta + \sigma^{2}(\theta)\beta\sigma(\beta) +  \ldots + \sigma^{n-1}(\theta)\beta\sigma(\beta) \cdots \sigma^{n-2}(\beta).
\end{equation}
If $\mathcal{N}_{L/\kk}(\beta) = 1$, this formula makes $\alpha/\sigma(\alpha) = \beta$ a formal identity.  One has to show that $\alpha \ne 0$ for some $\theta \in L$ to complete the proof.

In  \cite{cohn:advnothr}, Chapter XI, Theorem 2 there is a constructive proof of \eqref{e:ht90} for $n = 2$.  It actually exhibits a nonzero $\alpha$:  Substituting $\theta = 1$ gives the simplified formula
$$\alpha = 1 + \beta.$$
This gives a nonzero $\alpha$ unless $\beta = -1$, and this case is handled separately.

We note that $\theta=1$ is the \emph{only} nonzero value guaranteed to be in \emph{every} field.  We make this choice for arbitrary $n$. We assume $L/\kk$ is a finite Galois extension with Galois group $G$, and $\sigma \in G$ is of order $n$.  We impose the condition  on $r\in L$ that
\begin{equation}
\tag{M} \label{e:M}
1 + r + r\sigma(r)  + \ldots + r\sigma(r) \cdots \sigma^{n-2}(r) = 0.
\end{equation}
We call \eqref{e:M} the \emph{Murphy condition}.  Unfortunately, \eqref{e:M} alone does not force cyclic extensions.  In \S\ref{s:m4} we see that when $n=4$,  $L^{\sigma}$ can be a \emph{proper} extension of $\kk$, even with the additional conditions in Eq. \eqref{e:admiss}.

The following formal properties are immediate:
 \begin{prop}\label{p:norm1}
 Let $\kk, \,L, \sigma, \, r \,\text{and}\; n$ satisfy \eqref{e:M}.  Then 
\begin{enumerate}
\renewcommand{\labelenumi}{(\alph{enumi})}
\item  \eqref{e:M} holds if $r$ is replaced by $\sigma^{i}(r)$, $1\le i\le n-1$. 
\item  $r\sigma(r)\cdots\sigma^{n-1}(r) = 1$.
\end{enumerate}
 \end{prop}
 \begin{proof}
 For (a), apply $\sigma^{i}$ to \eqref{e:M}.  For (b), apply  $\sigma$ to \eqref{e:M}, multiply by $r$, and subtract \eqref{e:M}.
 \end{proof}
If $\kk$ is a number field, and $r$ is an \emph{algebraic integer} for which \eqref{e:M} holds, then by Prop.~\ref{p:norm1}(b), $r$ and its $\sigma$-conjugates are \emph{units}, which we call \emph{Murphy's units}.  They are quotients of unusual sets of conjugate associates (see the Remarks just before \S\ref{s:shenpolys}).

We take as a simple example the case $n=3$.  Then  \eqref{e:M} becomes
\begin{equation}
\tag{M3}\label{e:M3}
1 + r + r\sigma(r) = 0.
\end{equation}
Solving for $\sigma(r)$ and repeatedly applying $\sigma$, treating it as a field automorphism, we obtain the expressions
\begin{equation}
\tag{C3} \label{e:C3}
\sigma(r) = \frac{-r - 1}{r}, \; \sigma^{2}(r) = \frac{-1}{r + 1}, \; \sigma^{3}(r) = r.
\end{equation}
Setting
\begin{equation}
\tag{Tr3} \label{e:Tr3}
r+\sigma(r)+\sigma^2(r)= A, \; A \in \kk
\end{equation}
and clearing fractions, we find that any $r$ for which \eqref{e:M3} and \eqref{e:Tr3} hold is a zero of 
\begin{equation}\tag{P3}\label{e:P3}
p(x) = x^{3} - Ax^{2} - (A + 3)x - 1, \text{ for some }A\in\kk.
\end{equation}
By Proposition \ref{p:primelt}, \emph{every} cyclic cubic field extension has a defining polynomial of this form.  However, if we take $\kk=\Q$ and $A\in\Z$, $p(x)$ is irreducible $\pmod{2}$, so this restriction on the parameter produces a family of cyclic cubic fields in which $2$ remains inert, which is clearly \emph{not} true of all cyclic cubic fields.  In \cite{shanks:simpcub}, D. Shanks called the fields defined by \eqref{e:P3} with $A\in\Z$  the ``simplest'' cubic fields.  Certain families of cyclic number fields of degrees $4$, $5$, and $6$ have subsequently been dubbed ``simplest.''  They, too, have defining polynomials whose zeroes are units satisfying \eqref{e:M}.  We have the following result:
\begin{prop}\label{p:known}  Let $\kk$ be a field, $t\in\kk$.  In each of the following cases, $\sigma(x)$ \textup{(mod} $P(x)$\textup{)} makes  \eqref{e:M} a formal identity, with $n$ equal to the degree of $P(x)$.
\begin{align*}&\textup{a) } P(x)=x^{3} - tx^{2} - (t + 3)x - 1,\sigma(x)=(-x-1)/x;\\
&\textup{b) }P(x)=x^{4} - tx^{3} - 6x^{2} + tx + 1,\sigma(x)=(-x-1)/(x-1);\\
&\textup{c) }P(x)=x^4 + (t^2 + 2t + 4)x^3 + (t^3 + 3t^2 + 4t + 6)x^2\\
&+ (t^3 + t^2 + 2t + 4)x + 1,\\
&\sigma(x)= (-x^3 - (t^2 + 2t + 4)x^2 - (t^3 + 3t^2 + 4t + 5)x\\
&-t^3 - t^2 - 2t - 2)/t;\\
&\textup{d) }P(x)=x^{5} - t^{2}x^{4} - (2t^{3} + 6t^{2} + 10t + 10)x^{3}\\
& - (t^{4} + 5t^{3} + 11t^{2} + 15t + 5)x^{2}+ (t^{3} + 4t^{2} + 10t + 10)x - 1,\\
&\sigma(x)=((-t - 1)x^4 + (t^3 + 2t^2 + 3t + 3)x^3 \\
&+ (t^4 + 4t^3 + 9t^2 + 14t + 8)x^2- (t^4 + 7t^3 + 19t^2 + 29t + 19)x \\
&- t^4 - 6t^3 - 16t^2 - 20t - 9)/(t^3 + 5t^2 + 10t + 7);\;\text{and} \\
&\textup{e) }P(x)=x^{6} - 2tx^{5} -(5t + 15)x^{4} - 20x^{3} + 5tx^{2} + (2t + 6)x + 1,\\
&\sigma(x)=(-2x - 1)/(x - 1).
\end{align*}
\end{prop}
\begin{proof}
For (a), (b), and (e), the fact that \eqref{e:M} becomes a formal identity is easily checked by hand.  The others can be checked with symbolic algebra software.  The author (with Phil Carmody's guidance) used Pari-GP.
\end{proof}
\begin{rems}
The term ``formal identity'' means that substituting $\sigma(x)$ and its compositional powers into \eqref{e:M} gives a quotient of polynomials in which $P(x)$ divides the numerator.  Using the linear fractional transformations $\sigma(x)$ which are independent of $t$ in (a), (b), and (e) and their compositional powers, \eqref{e:M} evaluates to $0$ at any $x$ for which all the terms in the sum are defined.

Shanks's 1974 paper \cite{shanks:simpcub} seems to be the first to refer to certain families of number fields as ``simplest.''   The polynomials in (a) had previously appeared in H. Cohn's 1956 paper \cite{cohn:cubics}. They yield explicit systems of independent units, a property which Shanks sought.  When this system is fundamental, the regulator is very small.  A great deal of research has been done on these cubics.

In \cite{wash:cubicclassnum}, L.C. Washington extended Uchida's work in \cite{uchida:cubics}, to force class numbers divisible by $n$ in the ``simplest'' cubic fields, and used elliptic curves to describe the $2$-Sylow subgroup of their class groups, and to exhibit explicit quartic extensions of these fields.

In \cite{morton:cubics}, Patrick Morton gave a parameterization of cyclic cubic fields based on automorphism polynomials, and obtained Shanks's simplest cubics by change of parameter.  Robin Chapman simplified Morton's proofs in \cite{chapman:cubics}.

E. Thomas proved in \cite{thomas:fundunicub} that if $r$ is a zero of a Shanks's simplest  cubic, then $\langle r, r + 1\rangle$ is a system of fundamental units for the order $\Z[r]$.

The fact that $r$ and $r+1$ are both units means that  $-r$ and $r+1$ are \emph{exceptional units} (two units adding to $1$) by  Nagell's definition in \cite{nagell:unites}.  D. Buell and V. Ennola studied the only other family of totally real cubic fields with exceptional units in \cite{buellenn:parmquadcub}.    In \cite{lenstra:euclidean}, H. Lenstra defined \emph{exceptional sequences} as finite sequences (of algebraic integers), the difference of any two of which is a unit.  Here, $0$, $1$, $r+1$ is a ``normalized'' $3$-term exceptional sequence.

The polynomials in (b) were constructed in  \cite{gras:tableclassunit}, while those in (e) were introduced (in a slightly different form) in \cite{gras:specunit}.  The quintics in (d) are of the form $-f(-x)$ where $f(x)$ is one of the quintics introduced by E. Lehmer in \cite{lehmer:cyclunits}.  She observed in this paper that the zeroes of the cubic polynomials in \cite{shanks:simpcub} and the quartic and sextic polynomials in \cite{gras:tableclassunit} and \cite{gras:specunit}  are, in the case of a prime conductor, integer translates of Gaussian periods, and obtained quintic units with the same property.    In  \cite{schoofwash:quintpoly} (Appendix), R. Schoof and L.C. Washington proved that the integer-translates property characterizes the ``simplest'' fields of degrees $2$, $3$, and $4$ with prime conductor.  In \cite{lazarus:classunitquart}, A. Lazarus applied the term ``simplest'' to these fields of degrees $2$ to $6$, as well as to the candidate family of ``simplest'' octic fields constructed by Y.Y. Shen in \cite{shen:uniclassoct}.

In Proposition \ref{p:ptwash} we show that for $t\in\kk - \{0, -2\}$, the polynomial $P(x)$  in (c), which falls out from \eqref{e:M} and the conditions in Eq. \eqref{e:admiss} with $n=4$, defines the same extension of $\kk$  as the polynomial 
\begin{equation}\label{e:washpolys}
f_{t}(x) = x^{4}-t^{2}x^{3}-(t^{3}+2t^{2}+4t+2)x^{2}-t^{2}x+1
\end{equation}
which L.C. Washington constructed in \cite{wash:quartics}.  In \cite{morton:quartics}, Patrick Morton proved the equivalence of these fields with cyclic quartic fields whose Galois groups have a quadratic (rather than cubic) generating automorphism polynomial.

In \cite{shanks:simpcub} Shanks also proposed ``simplest'' quadratic fields defined by the polynomials $x^{2}=ax+1$, $a\in\Z$.  These quadratic fields have a special significance with respect to \eqref{e:M}.  For if $v^{2}-tv - 1=0$, $t\in\Z-\{0,-2\}$, the cyclic quartic field $L$ defined by $f_{t}(x)$ contains $v$, and if $G(L/\Q)=\langle\sigma\rangle$, we have
\begin{equation}\label{e:normneg1}
1+v+v\sigma(v)+v\sigma(v)\sigma^{2}(v)=1 + v + (-1) +(-1)v=0.
\end{equation}
That is, quadratic units of norm $-1$ (unlike those of norm $+1$), satisfy \eqref{e:M} with $n=4$ when embedded in a cyclic quartic field (such as a Washington's cyclic quartic field).  We give a generalization of this phenomenon based on \eqref{e:M}, in the Remarks just before \S \ref{s:shenpolys}.

The ``simplest'' number fields of degrees 3, 4, and 6 were studied further by G. Lettl, A. Peth\H{o}, and P. Voutier in \cite{lettpetvou:simplethue} and \cite{lettpetvou:arith}.  A. Lazarus studied the unit groups and class numbers of the ``simplest'' quartic fields  in \cite{lazarus:classunitquart}.  

L. C. Washington used coverings of modular curves in \cite{wash:quartics} to construct his family of cyclic quartic fields.  He observed that the ``simplest'' fields of degree $2$, $3$, $4$, $5$, and $6$ can be constructed by the same method.   (See  \cite{darmon:note} with regard to Emma Lehmer's quintics.)  He further observed that in all these cases, the coverings have genus $1$.  In contrast, the cyclic sextic fields constructed by O. Lecacheux in \cite{lecacheux:sextics}, use a covering of genus $2$.

In \cite{louboutin:effcomp}, S. Louboutin obtained explicit formulas for powers of Gaussian sums attached to the ``simplest'' fields of degrees $2$ to $6$, and used them to give efficient computations of  their class numbers.
\end{rems}
In \S\ref{s:shenpolys}, we construct a $1$-parameter family of polynomials of degree $n$ in $\Q(\cot(\pi/n))[x]$, whose zeroes are permuted cyclically by a linear fractional transformation $\lambda$ of compositional order $n$ for each $n>1$, generalizing a construction in \cite{shen:uniclassoct}.  When $n=4m$, $\sigma=\lambda^{-1}$ makes \eqref{e:M} a formal identity.

For $n>3$ we cannot simply ``solve'' \eqref{e:M}  for $\sigma$ as we did in \eqref{e:C3}; but we  would like to obtain as general an algebraic map of compositional order $n$ as possible, that makes \eqref{e:M} a formal identity.  So,  we impose purely  algebraic conditions which always hold when $\kk(r)/\kk$ is cyclic of degree $n$ with $G=\langle\sigma\rangle$:
\begin{equation}\label{e:admiss}
\sum_{i = 0}^{\frac{n}{d}-1}\sigma^i\left(\prod_{j=0}^{d-1}\sigma^{\frac{jn}{d}}(r)\right) \in \kk,\text{ for each divisor }d\text{ of } n.
\end{equation}
These conditions do not depend on the choice of cyclic generator for $\langle\sigma\rangle$ but (as mentioned after Proposition \ref{p:known}), do admit extensions in which $L^{\sigma}$ is a \emph{proper} extension of $\kk$.

In \S\ref{ss:formconst}, using only \eqref{e:M} for $n=4$, the condition that the map $\sigma$ fixes the ground field elementwise, the conditions \eqref{e:admiss}, and elementary algebra, we obtain a linear fractional map \eqref{e:C4} for $\sigma$, and a 2-parameter family of  monic octics $T(m,A,x)$ ($m,A\in\kk$),  with ``generic'' Galois group $G\cong$ $_{8}T_{11}$ (see below).  We let $L/\kk$ denote the splitting field of $T(m,A,x)$.

$L/\kk$ contains (\S\ref{ss:quadfacts}) the elementary Abelian extension $E=\kk(s,w,y)$ of $\kk$, where $s^{2}=m^{2}-4$, $w^{2}=(m+2+A)^2-4(m-2)$, and $y^{2}=A^{2}-4(m-2)$. When  $[E\!:\!\kk]=8$, $G \cong$ $ _{8}T_{11}$, and $E$ contains $7$ quadratic extensions of $\kk$.  In Theorem \ref{t:galgp1} we describe the subgroups of $_{8}T_{11}$ fixing each of these.

The description of the Galois group is greatly facilitated by the fact that  the \emph{related} octics $T(m,A,x)$ and $T(m,-m-2-A,x)$ have the same splitting field over $\kk$ (Theorem \ref{t:relocts}).  When $[L\!:\!\kk]=8$, at least one of these octics is a defining polynomial for the splitting field.

The cyclic quartics for which our  map $\sigma$ makes \eqref{e:M} a formal identity with $n=4$, occur in \emph{pairs}.  They are the quartic factors of $T(m,A,x)$ in $\kk(sw)[x]$, so typically define cyclic quartic extensions of $\kk(sw)$, not of $\kk$.  But when $sw\in\kk$, they are in $\kk[x]$, and we call them \emph{Murphy's twins}. They ``generically'' define distinct cyclic quartic extensions of $\kk$, both containing $\kk(s)$.  The ``simplest'' quartic fields and Washington's cyclic quartic fields are ``degenerate'' cases (with $\mathbf{k}=\mathbb{Q}$) for which $s=0$ and $w=0$, respectively.  In these cases the ``twins'' are identical.  We construct other Murphy's twins extensions of $\mathbb{Q}$ in \S \ref{ss:murphtwins} using standard results on norms from real quadratic fields.

By ``collapsing'' other quadratic extensions of $\kk$ in $E$, we construct families of polynomials defining normal octic extensions with $G\cong D_{8}(8)$, $Q_{8}$, and $C_{4}\times C_{2}$; and quartics with $G\cong D_{4}$ and $V_{4}$.

The ``generic'' Galois group $_{8}T_{11}$ is an order-16 transitive subgroup of $A_{8}$; $_{8}T_{11}\cong$ GAP small group $\langle 16,13\rangle$.  It has one maximal subgroup $\cong Q_{8}$ (quaternion group), three  $\cong D_{4}$, and three $\cong C_{4} \times C_{2}$.  It has presentation
$$\langle a,b,c\rangle:\;a^{4}=b^{2}=c^{2}=1,ab=ba,\;ac=ca,\;(bc)^{2}=a^2.$$

It is also known as the ``almost extraspecial group of order 16.''  Derek Holt (\cite{holt}) described almost extraspecial $p$-groups as \emph{central products}.  A central product is an ``amalgamated product'' (see \cite{hall:thrgroups}), in which the subgroups being identified are in the centers of the factors.  In this particular case,
$$_{8}T_{11} \cong C_{4} \vee D_{4} \cong C_{4} \vee Q_{8}.$$

Our octics produce (Propositions \ref{p:unitassoc} and \ref{p:constell}) some unusual sets of units and associates.  Specifically (Proposition \ref{p:unitassoc}(b)), if $m,A\in\mathcal{O}_{\kk}$ for a number field $\kk$, and $m-2\in\mathcal{O}_{\kk}^{\times}$, then each zero of $T(m,A,x)$ is part of an \emph{exceptional sequence of three units}; that is, three units, the difference of any two of which is also a unit.  In particular, $T(1,A,x)$ and $T(3,A,x)$, $A\in\Z$, are one-parameter families in $\Z[x]$ whose zeroes have this property.  An infinite subfamily of $T(3,A,x)$ gives  (Eqs. \eqref{e:sw1} - \eqref{e:barsw2}) Murphy's twins  cyclic quartic fields with exceptional sequences of three units.  When $\kk=\Q$ and $m,A\in\Z$, there are (\S\ref{ss:regs}) explicit sets of $3$ independent units.  If $T(m,A,x)$ defines one or more number fields whose units groups have rank $3$, these can produce rather small regulators.  We also obtain (Eq. \eqref{e:unitindex}) a unit index formula which may be of interest.

The terms ``exceptional units,'' ``exceptional sequences,'' and ``cliques'' (of units) arose as follows: In \cite{nagell:unites}, Nagell called either of two units whose sum is 1, \emph{exceptional units}.  In \cite{lenstra:euclidean}, Lenstra used finite sequences of algebraic integers, the difference of any two of which is a unit, to construct Euclidean fields.  He observed that by applying an appropriate affine transformation, the first two terms of the sequence become 0 and 1, and any further terms become exceptional units.  Thus, Lenstra's sets generalize the concept of exceptional units.  In \cite{leutmart:lenstconst}, A. Leutbecher and J. Martinet called them \emph{exceptional sequences}.  In \cite{leut:euclenstra}, Leutbecher entitled Section 1 ``The graph of exceptional units.''  This applied an idea of Gy\H{o}ry, that defining two elements of a commutative ring $R$ as being ``connected'' when their difference is a unit, induces a graph structure on $R$ (see, e.g.\ \cite{leutnikl:cliquesexcuni}, reference [G3]).  Leutbecher and G. Niklasch used the graph-theoretic term ``cliques'' in this context in \cite{leutnikl:cliquesexcuni}.
%
%
\section{Basic formal properties}\label{s:basic}
Proposition \ref{p:norm1} gave some very simple formal properties implied by \eqref{e:M}.  We give several more.  The first of these applies to non-Abelian extensions.
\begin{prop}\label{p:nonabel}
Let $L/\kk$ be a finite Galois extension with Galois group $G$, and assume $\sigma\in G$, $r$,and $n$ satisfy \eqref{e:M}.  If $\gamma \in C_{G}(\langle\sigma\rangle)$, then \eqref{e:M} holds for $\gamma(r)$.
\end{prop}
\begin{proof}Since $\gamma\sigma=\sigma \gamma$, we have $\gamma(\sigma^{i}(r))=\sigma^{i}(\gamma(r))$ for all $i$.
\end{proof}
Next, \eqref{e:M} does not depend on the choice of generator for $\langle\sigma\rangle$.
\begin{prop}\label{p:diffgen}
Assume $\kk, \,L,r, \sigma,\,\text{and}\; n>2$ satisfy \eqref{e:M}.   If \\$1 \le k  < n, \, (k, n) = 1$,  $\eta = \sigma^{k}$, and $y_{k}= r\sigma(r) \cdots \sigma^{k-1}(r)$, then 
$$1 + y_{k} + y_{k}\eta(y_{k}) + \ldots + y_{k}\eta(y_{k}) \cdots \eta^{n-2}(y_{k}) = 0.$$
\end{prop}
\begin{proof}
The proof is left as an exercise.
\end{proof}
The $y_{k}$ are zeroes of  the degree-$n$ polynomials
\begin{equation}\label{e:diffpols}f_{k}(x)=\prod_{i=1}^{n}(x-\sigma^{i}(y_{k})).
\end{equation}
We thus obtain a set of $\varphi(n)$ polynomials.  Mutually inverse generators of $\langle\sigma\rangle$ give polynomials with mutually reciprocal zeroes.
\begin{exmps}
Applying Eq.\ \eqref{e:diffpols} to $f_{1}(x)=P(x)$ and $\sigma$ as in Proposition \ref{p:known}(d), we obtain the alternate defining polynomial
\begin{align*}&f_{2}(x)=x^5 + (t^3 + 3t^2 + 5t + 5)x^4 - (t^4 + 3t^3 + 7t^2 + 5t + 5)x^3\\
&- (t^4 + 5t^3 + 17t^2 + 25t + 25)x^2 - (2t^2 + 5t + 10)x - 1.\end{align*}
Eq.\ \eqref{e:diffpols} gives $f_{2}(x)$, $f_{3}(x)=-x^{5}f_{2}(1/x)$, and $f_{4}(x)=-x^{5}f_{1}(1/x)$ in addition to $f_{1}(x)$.

The zeroes of the octic polynomials studied by Y.Y. Shen in \cite{shen:uniclassoct},
$$P(a,x)=x^8 - ax^7 - 28x^6 + 7ax^5 + 70x^4 - 7ax^3 - 28x^2 + ax + 1,\;a\in\Z,$$
are permuted cyclically by the compositional powers of the algebraic map
$$\sigma: x\mapsto (-\xi x-1)/(x-\xi),\text{ taking }\xi\pmod{\xi^{2}-2\xi-1}.$$
It may be verified directly that \eqref{e:M} with $n=8$ is a formal identity for the compositional powers of this algebraic map.  Shen showed that $P(a,x)$ defines cyclic octic fields when $a^{2}+64\in2\Z^{2}$, and investigated the properties of these fields.  For such $a$, $P(a,x)$ is irreducible in $\Q[x]$ but \emph{not} in $\Q(\xi)[x]$.  In this case, the \emph{automorphism} defined by $\sigma$ maps $\xi$ to $-1/\xi$, so its compositional powers occur in a different order than those of the algebraic map.  If we assume that $\kk=\Q(\xi)$, $a\in\mathcal{O}_{\kk}$, and $P(a,x)$ is irreducible in $\kk[x]$, then it defines a cyclic octic extension of $\kk$ with Galois group $\langle\sigma\rangle$, and its zeroes are \emph{units} which satisfy \eqref{e:M} with $n=8$.  The coefficients of  $f_{3}(x)$ and $f_{5}(x)$ as in Eq. \eqref{e:diffpols} are not formally in $\Z[a]$; for instance, the coefficient of $x^{5}$ in $f_{3}(x)$ is $(12a^2 + 768)\xi + a^3 - 12a^2 + 57a - 768$.
\end{exmps}
From Proposition \ref{p:norm1}(b) and \ref{e:ht90}, we see that if $L/\kk$ is cyclic of degree $n$ with Galois group $\langle\sigma\rangle$, any $r$ for which \eqref{e:M} holds is of the form $\alpha/\sigma(\alpha)$, $\alpha\in L$.  If $\alpha=1/z$, we have $r=\sigma(z)/z$.  Substituting this expression into \eqref{e:M}, the products in each term telescope, giving
$$(z + \sigma(z)+\ldots+\sigma^{n-1}(z))/z=0.$$
The numerator is $\mathbf{Tr}_{L/\kk}(z)$.  Thus, in a cyclic extension, an element which satisfies \eqref{e:M} is of the form $\sigma(z)/z$ where $z$ is \emph{in the kernel of the trace}.  For cyclic extensions of degree $n>2$, primitive elements of this form always exist.
\begin{prop}\label{p:primelt}
Let $L/\kk$ be a cyclic extension of degree $n$ with $G=\langle\sigma\rangle$.  If $n>2$, there is $r\in L$ with $L=\kk(r)$ for which \eqref{e:M} holds.
\end{prop}
\begin{proof}
The kernel of the trace from $L$ to $\kk$ is a $\kk$-vector subspace $V$ of $L$ of dimension $n-1$.  The distinct elements $\sigma(z)/z$, $z\in V-\{0\}$, correspond in an obvious way to the $1$-dimensional $\kk$-subspaces of $V$.

If $\kk$ is a finite field with $q$ elements, the number of such subspaces is
$$(q^{n-1}-1)/(q-1) = 1+q+\ldots+q^{n-2}.$$
The usual formula for the number of monic irreducible polynomials of degree $n$ in $\kk[x]$ shows that there are fewer than
$$1 + q + \ldots +q^{n/\ell}$$
nonzero \emph{non}primitive elements for $L/\kk$, where $\ell$ is the least prime factor of $n$.  If $n>2$, the number of distinct elements $\sigma(z)/z$, $z\in V-\{0\}$, is thus too large for them all to be non-primitive.

Now suppose $\kk$ is infinite and $n>2$.  Let $d\mid n$, $d>1$, and suppose $z\in V$ with $\sigma(z)/z=w\in F$, where $F$ is the intermediate field with $L/F$ of degree $d$.  Then $\mathcal{N}_{F/\kk}(w)=\sigma^{n/d}(z)/z=W$; $\mathcal{N}_{L/\kk}(W)=W^n=1$; there are at most $n$ such $W$.  The distinct values of $\sigma^{n/d}(z)/z$ correspond to one-dimensional $F$-vector subspaces of $L$; these are $n/d$-dimensional $\kk$-vector spaces.  Now $n/d<n-1$  for $d>1$ and $n>2$. Thus, the values $r=\sigma(z)/z$, $z\in V$, for which $\kk(r)\neq L$, correspond to the $1$-dimensional subspaces of a finite union of \emph{proper} $\kk$-subspaces of $V$.  It is well-known that when $\kk$ is infinite, no $\kk$-vector space is a finite union of proper subspaces, and the result follows.
\end{proof}
\begin{rems}
If $r = \sigma(z)/z$ satisfies \eqref{e:M}, then in Proposition \ref{p:diffgen}, $y_{k} = \sigma^{k}(z)/z$.

When the ground field $\kk$ contains a primitive $m$th root of unity $\zeta\ne 1$, $F=\kk(r)$ is a cyclic extension of degree $d$, and $\mathcal{N}_{F/\kk}(r)=\zeta$, then embedding $F$ in a cyclic extension of degree $md$ over $\kk$ forces $r$ to satisfy \eqref{e:M} with $n=md$, as occurs in Eq.\ \eqref{e:normneg1} with $m=2$ and $\zeta=-1$.

If $\kk$ is a number field and $r$ is a ``Murphy's unit,'' we may take $z$ so its $\sigma$-conjugates are  \emph{associate algebraic integers in the kernel of the trace}.  Among cyclic cubic extensions of $\Q$, only Shanks's simplest cubic fields possess such $z$.  If $P(r)=0$ in Prop.~\ref{p:known}(a), $r=\sigma(z)/z$ for $z=2r^{2} - (2t+ 1)r -t - 4$.  Noting that $(t^2+3t+9)^{2}=\text{disc}(x^{3} - tx^{2} - (t + 3)x - 1)$, we have
$$z^3-(t^2 + 3t + 9)z+t^2 + 3t + 9=0.$$
\end{rems}
%
%
\section{Shen's polynomials}\label{s:shenpolys}
In \cite{shen:uniclassoct}, Y.Y. Shen constructed polynomials of $2$-power degree, generalizing the $1$-parameter family of defining polynomials for the octic fields he investigated.  His construction used the polynomials
\begin{equation}\label{e:shen1}
Q_{n}(x)=\Re((x+i)^{n})\text{ and }V_{n}(x)=\Im((x+i)^n).
\end{equation}
His work (Proposition 3(b) and (c) of \cite{shen:uniclassoct}) shows that for $n=2$ and $n=4$, the polynomials
\begin{equation}\label{e:shen0}
P_{n}(a,x)=Q_{n}(x)-\frac{a}{n}V_{n}(x)
\end{equation}
give the usual defining polynomials for the ``simplest'' quadratic and quartic fields.  He obtained his octic fields by investigating \eqref{e:shen0} for $n=8$, and observed that taking $n=2^k$ gives a $1$-parameter family with similar properties.

Here, we generalize Shen's construction to polynomials of degree $n$ for all $n>1$. Clearly
$$Q_{n}(x)=\sum_{0\le 2k\le n}(-1)^{k}\binom{n}{2k}x^{n-2k}\text{ , }V_{n}(x)=\sum_{0\le 2k<n}(-1)^{k}\binom{n}{2k+1}x^{n-2k-1}$$
so $Q_{n}(x)$ is monic of degree $n$, and $V_{n}(x)$ is of degree $n-1$, with leading coefficient $n$. The $\gcd$ of the coefficients of $V_{n}(x)$ is $2^{v_{2}(n)}$, where $n/2^{v_{2}(n)}$ is an odd integer.  The proof is left as an exercise.  We take
\begin{defn}\label{d:coeffs}With $Q_{n}(x)$ and $V_{n}(x)$ as in Eq.~\eqref{e:shen1}, and $n\in\Z^{+}$, set
\begin{equation}\label{e:shenpolys}
P_{n}(a,x)=Q_{n}(x)-\frac{a}{2^{v_{2}(n)}}V_{n}(x).
\end{equation}
\end{defn}
Then $P_{n}(a,x)$ is monic of degree $n$.  For $0\le k\le n$, the coefficient of $x^{n-k}$ is in $\Z$ if $k$ is even, and in $\Z\cdot a$ if $k$ is odd.  Clearly $V_{2n}(x)=2Q_{n}(x)V_{n}(x)$, so $V_{2^{t}}(x)=2^{t}Q_{1}(x)Q_{2}(x)\cdots Q_{2^{t-1}}(x)$ for $t\in\Z^{+}$ by induction.  Thus, when $n=2^{t}$, $P_{n}(a,x)$ coincides with the degree-$2^{t}$ polynomial constructed in \cite{shen:uniclassoct}.
For each $n>1$, Definition \ref{d:coeffs} gives a $1$-parameter family of monic polynomials in $\Z[a][x]$.  The useful formula
\begin{equation}\label{e:multcot}
 Q_{n}(\cot(\theta))/V_{n}(\cot(\theta))=\cot(n\theta)\text{ for all }n\in\Z^{+}
\end{equation}
follows directly from Eq. \ref{e:shen1}.  Taking $x=\cot(\theta)$, we see that the zeroes of $P_{n}(a,x)$ are permuted cyclically by the linear fractional transformation $\lambda$ in \eqref{e:twomaps}.
The results in this section are thus very similar to those in \cite{rikuna:simplecyclic}, but are less general.  We assume our parameter to be in the complex field $\C$.  This allows us to ``cheat'' by invoking properties of $\cot(\theta)$ as a periodic  meromorphic function, and to exploit the fact that  $\lambda$ corresponds to adding a division point to the argument of this function.  Also, we assume the ground field (field of coefficients of $\lambda$) contains $K=\Q(\cot(2\pi/n))$ rather than $k=\Q(\cos(2\pi/n))$ as in \cite{rikuna:simplecyclic}.   Restricting our parameter to $\mathcal{O}_K$, produces polynomials of degree $n$ in $K[x]$ which define cyclic extensions of $K$ with several formal properties like those of the ``simplest'' quartic fields (Theorem \ref{t:shenmain}).  It is easily shown that $K=\Q(\cos(2\pi/N))$ where $N=\textup{lcm}(n,4)$.  Thus, $K= k$  as in \cite{rikuna:simplecyclic} precisely when $4\mid n$.  In this case, Proposition \ref{p:shenht90} shows that \eqref{e:M} holds with $\sigma=\lambda^{-1}$. The irreducible factors of $P_{n}(a,x)$ in $K[x]$ all have the same degree $d\mid n$.  If $d$ is a proper divisor of $n$, the ``rescaled'' polynomial of degree $n/d$ splits into linear factors in $K[x]$ (Prop. \ref{p:compsplit}).

The following properties are easily obtained:
\begin{prop}\label{p:polyidents}
Let $n\in\Z^{+}$ and  $a\in\C$.
\begin{enumerate}
\renewcommand{\labelenumi}{(\alph{enumi})}
\item  If $a\ne \pm 2^{v_{2}(n)}i$, $P_{n}(a,x)$ has $n$ distinct zeroes $x=\cot(\theta)$ for which $\cot(n\theta)=a/2^{v_{2}(n)}$.  The zeroes are real if $a$ is real.  If $\cot(\theta_{1})$ is one of them, all are given by $\cot(\theta_{1}+k\pi/n)$, $0\le k\le n-1$.
\item The discriminant of $P_{n}(a,x)$ is $n(2^{n-2-2v_{2}(n)}n)^{n-1}(a^{2}+4^{v_{2}(n)})^{n-1}$.
\item If $P_{n}(a,\cot(\theta))=0$, then $Q_{k}(\cot(\theta))-\cot(k\theta)V_{k}(\cot(\theta))=0$ for $k\in\Z^{+}$.
\item If $P_{n}(a,\cot(\theta))=0$ and $D\mid n$, $P_{n/D}(a/2^{v_{2}(D)},x)=0$ for
$$x=\cot(D\theta+kD\pi/n),\; 0\le k \le n/D-1.$$
\end{enumerate}
\end{prop}
\begin{proof}
For (a), let $z=a/2^{v_{2}(n)}$.  For $z\in\C-\{\pm i\}$ there is $\theta$ for which  $\cot(n\theta)=z$;  we have $(0,\pi/n)\mapsto \R$. Then, $\cot(n\theta)$ has period $\pi/n$.

For (b), disc$(P_{n}(a,x))=f(a)\in\Z[a]$.  By construction, $P_{n}(a,x) = (x\mp i)^{n}$ when $a=\pm 2^{v_{2}(n)}i$.  By (a), $f(a)$ has no linear factors in $\C[a]$ other than $a\pm 2^{v_{2}(n)}i$, so is a constant multiple of  $(a^{2}+4^{v_{2}(n)})^{n-1}$.  Taking $a=0$, we see the constant is the discriminant of $Q_{n}(x)$ divided by $4^{(n-1)v_{2}(n)}$.  The identity $Q_{n}'(x)=nQ_{n-1}(x)$ aids in evaluating the discriminant of $Q_{n}(x)$.

Finally, (c) and (d) follow easily from Eq. \eqref{e:multcot}.
\end{proof}
We recover the following well-known identities  by comparing the coefficient of $x^{n-1}$  and the constant term of $P_{n}(a,x)$ with the formula $\cot(n\theta)=a/2^{v_{2}(n)}$ in Proposition \ref{p:polyidents}(a).  They are valid if $n\theta$ is not an integer multiple of $\pi$.
\begin{subequations}
\begin{equation}\label{e:cotsum}
\text{If }n\in\Z^{+},\text{ then }\sum_{k=0}^{n-1}\cot(\theta+k\pi/n)=n\cot(n\theta).
\end{equation}
\begin{equation}\label{e:cotprod}
\text{If } n\in\Z^{+}\text{ is odd, then }\prod_{k=0}^{n-1}\cot(\theta+k\pi/n)=(-1)^{(n-1)/2}\cot(n\theta).
\end{equation}
\end{subequations}
Let $n>1$, $x=\cot(\theta)$ and $\xi = \cot(\pi/n)$.  Then the map
\begin{equation}\label{e:twomaps}
T:\;\cot(\theta)\mapsto\cot(\theta+\pi/n)\text{ becomes }\lambda :\; x\mapsto  \frac{\xi x-1}{x+\xi}.
\end{equation}
Since $T$ has compositional order $n$, so does $\lambda$.  Using Eq.\ \eqref{e:multcot} and Proposition \ref{p:polyidents}(a), we have
\begin{equation}\label{e:sumlft}
\sum_{k=0}^{n-1}\lambda^{(k)}(x)=nQ_{n}(x)/V_{n}(x),
\end{equation}
where  $\lambda^{(k)}$ is the $k^{th}$ compositional power of $\lambda$. Also,
\begin{cor}\label{c:cycperm}
Let $n\in\Z$, $n>1$, and $a\in\C$.
If $\xi=\cot(\pi/n)$, the zeroes of $P_{n}(a,x)$ are permuted cyclically by the linear fractional map $\lambda$ in \textup{Eq.} \eqref{e:twomaps}.
\end{cor}
\begin{proof}
If $a\ne\pm 2^{v_{2}(n)}i$, this just restates Proposition \ref{p:polyidents}(a). If $a=\pm 2^{v_{2}(n)}i$, the $n$-fold zero $\pm i$ of $P_{n}(a,x)=(x\mp i)^{n}$ is fixed by $\lambda$.
\end{proof}
Now, let $n\in\Z$, $n>1$ and $K=\Q(\cot(\pi/n))$.   We have the following result:
\begin{thm}\label{t:shenmain}
Let $n\in\Z$, $n>1$, $K=\Q(\cot(\pi/n))$, and $\alpha\in\mathcal{O}_{K}$.
\begin{enumerate}
\renewcommand{\labelenumi}{(\alph{enumi})}
\item The zeroes of $P_{n}(\alpha,x)$ are algebraic integers (units, if $n$ is even).
\item If $L/K$ is the splitting field of $P_{n}(\alpha,x)$, then $L/\Q$ is totally real.
\item There is a divisor $d$ of $n$ such that, if $f(x)$ is any irreducible factor of $P_{n}(\alpha,x)$ in $K[x]$, $f(x)$ has degree $d$.
\item If $f(x)$ and $d\mid n$ are as in \textup{(c)}, then $G(L/K)=\langle\sigma\rangle$, where  the restriction of $\sigma$ to the zeroes of $f(x)$ is given by $\lambda^{(n/d)}$, $\lambda$ as in \textup{Eq}.\ \eqref{e:twomaps}.
\end{enumerate}
\end{thm}
\begin{proof}
For (a), $P_{n}(\alpha,x)$ is a monic polynomial in $\mathcal{O}_{K}[x]$ by Definition \ref{d:coeffs}.  If $n$ is even, the constant term is $(-1)^{n/2}$.

For (b), the zeroes of $P_{n}(\alpha,x)$ are all real by part (a) of Proposition \ref{p:polyidents}.  Since $K$ is totally real, the zeroes of $P_{n}(\alpha',x)$ are also all real, for each conjugate $\alpha'$ of $\alpha$ in $K/\Q$.

For (c), let $P_{n}(\alpha,r)=0$.  By Corollary \ref{c:cycperm}, the rest of the zeroes are given by $\lambda^{(k)}(r)$, $1\le k\le n-1$, $\lambda$ as in Eq.\ \eqref{e:twomaps}.  Now  $\lambda$ is defined by a matrix in $\mathbf{PGL}_{2}(K)$, so all zeroes of $P_{n}(\alpha,x)$ determine the \emph{same} extension of $K$. Thus, its irreducible factors in $K[x]$ all have the same degree.

For (d), let $f(x)$ be a monic irreducible factor of $P_{n}(\alpha,x)$ in $K[x]$, and $d\mid n$ its degree.   Let $f_{k}(x)  =(x+\cot(k\pi/n))^{d}f(\lambda^{(k)}(x))$ ``made monic'' (divided by its lead coefficient).

If $f_{j+k}(x)= f_{j}(x)$, clearly $\lambda^{(k)}$ has compositional order $d$ at most.  So $f(x)=f_{0}(x)$, $f_{1}(x),\ldots$, $f_{n/d-1}(x)$  are all distinct, and $f_{n/d}(x)=f(x)$.  The rest is now self-evident.
\end{proof}
The integer $d$ in Theorem \ref{t:shenmain}(c) has some additional properties:
\begin{prop}\label{p:compsplit}
Let $n$, $\alpha$, and  $P_{n}(\alpha,x)$ be as in \textup{Theorem \ref{t:shenmain}}, with $d$ as in  \textup{Theorem \ref{t:shenmain}(c)}.
\begin{enumerate}
\renewcommand{\labelenumi}{(\alph{enumi})}
\item If $P_{n}(\alpha,\cot(\theta))=0$, then $\cot(d\theta)\in K$.
\item $P_{n/d}(\alpha/2^{v_{2}(d)},x)$ splits into linear factors in $K[x]$.
\item $d$ is the least positive integer for which \textup{(a)} and \textup{(b)} hold.
\end{enumerate}
\end{prop}
\begin{proof}
Let $f(x)$ be the minimum polynomial for $\cot(\theta)$ in $K[x]$.  By Theorem \ref{t:shenmain}(d), the zeroes of $f(x)$ are $\cot(\theta+k\pi/d)$, $0\le k<d$.  By the identity in Eq. \eqref{e:cotsum}, the sum of these zeroes is $d\cot(d\theta)$, proving (a).  We have $f(x)=Q_{d}(x)-\cot(d\theta)V_{d}(x)$ by Prop. \ref{p:polyidents}(c); $d$ is the least positive integer with $\cot(d\theta)\in K$ since $f(x)$ is irreducible in $K[x]$.  By the formula in Prop. \ref{p:polyidents}(d), the zeroes of $P_{n/d}(\alpha/2^{v_{2}(d)},x)$ are $\cot(d\theta+k\pi d/n)$, $0\le k <n/d$.    These are all in $K$ because $\cot(d\theta)\in K$ and $\cot(k\pi d/n)\in K$.

\end{proof}
When $4\mid n$, $\langle\lambda\rangle$ has generators which make \eqref{e:M} a formal identity.
\begin{prop}\label{p:shenht90}
Let $\lambda$ be as in \textup{Eq.~\eqref{e:twomaps}}, $m\in\Z^{+}$, and $n=4m$.  If
$$S(x) = 1 + x + x\lambda^{-1}(x) + x\lambda^{-1}(x)\lambda^{-2}(x) +\ldots +x\lambda^{-1}(x)\cdots \lambda^{-(4m-2)}(x) ,$$
then $S(x)\equiv 0$; that is, if $\sigma^{k}(r)=\lambda^{-k}(r)$, \textup{\eqref{e:M}} becomes a formal identity.
\end{prop}
\begin{proof}
$S(x)$ is a rational function, which clearly has a partial fraction decomposition of the form
$$S(x)=A+Bx+\sum_{k=1}^{4m-2}\frac{c_{k}}{(x-\cot(k\pi/4m))},$$
where $A,B,c_{k}\in\Q(\cot(\pi/4m))$.  Using $\lambda^{2m}(x)=-1/x$, we find
$$B = \sum_{k=0}^{2m-1}\prod_{j=1}^{k}(-\cot(j\pi/4m)).$$
Since $\cot(\pi/2-\theta)=1/\cot(\theta)$, the terms of index $k$ and $2m-1-k$ are equal and opposite.  Thus $B=0$, so $S(x)$ is bounded as $x\rightarrow\infty$.

Now $\lambda^{-(4m-1)}(x)=\lambda(x)$.  Using $\lambda^{j+2m}(x)=-1/\lambda^{j}(x)$, we find that $\lambda(x)S(x)=S(\lambda(x))$.  Letting $x\rightarrow\infty$, we see $S(\lambda(x))$ remains bounded as $\lambda(x)\rightarrow\cot(\pi/4m)$, so $c_{1}=0$.  Then, $\lambda^{(k)}(x)S(\lambda^{k-1}(x))=S(\lambda^{(k)}(x))$, giving $c_{k}=0$ for $1\le k\le 4m-2$.  Thus, $S(x)=A$, a constant.  Substituting $x=i$, which is fixed by $\lambda$, we find $S(i)=0$.
\end{proof}
Next, we evaluate the sums in Eq.\ \eqref{e:admiss} with $\sigma=\lambda$.  If $d$ (and hence $n$) is even, the factors in each term occur in negative-reciprocal pairs, so the sum is $(-1)^{d/2}(n/d)$.  If $d$ is odd, we use Eq.\ \eqref{e:cotprod} to evaluate each term, then Eq.\ \eqref{e:cotsum} to add them; the result is $(-1)^{(d-1)/2}na/(d\cdot 2^{v_{2}(n)})$.

Let $n=4m$, $a\in\Z$. Since $\cot(\pi/n)\in\Q(r,\lambda(r))$ if $P_{n}(a,r)=0$, the splitting field of $P_{n}(a,x)$ over $\Q$ has degree divisible by $\varphi(2m)$.  If $\Q(r)/\Q$ is normal of degree $n$, then  $\varphi(2m)\mid 4m$. If $d$ is the degree over $K$, then $d$ divides $4m/\varphi(2m)$.  If $n=2^{t}$, $t\ge 4$, then $d\mid 4$, and it can be shown using Prop. \ref{p:compsplit}(b) that $\Q(r)/\Q$ is a normal extension of degree $n=2^{t}$ only if  $a=\pm n$ and $\Q(r)=\Q(2\cos(\pi/4n))$, or if $n=16$ and $a=\pm 16\cdot 239$.
%
%
\section{The condition \eqref{e:M} with $n = 4$}\label{s:m4}
Assuming that  \eqref{e:M} and the conditions in Eq.\ \eqref{e:admiss} hold, and that the map $\sigma$ fixes the ground field $\kk$ elementwise, we only need elementary algebra, treating $\sigma$ as a field isomorphism as needed, to produce a complete description of $\sigma$ as an algebraic map, in terms of two parameters $m$ and $A$ in $\kk$.
%
\subsection{The formal construction}\label{ss:formconst}
When $n = 4$, \eqref{e:M} becomes
\begin{equation}
\tag{M4} \label{e:M4}
1 + r + r\sigma(r) + r\sigma(r)\sigma^{2}(r) = 0.
\end{equation}
The condition of Eq.~\eqref{e:admiss} with $d = 2$ is
$$r\sigma^{2}(r) + \sigma(r)\sigma^{3}(r) = m \in \kk.$$
By Proposition~\ref{p:norm1}, $r\sigma(r)\sigma^{2}(r)\sigma^{3}(r) = 1$, so taking $r\sigma^{2}(r)= u$, we have
\begin{subequations}
\begin{equation}\label{e:sigmau}
\sigma(u) = u^{-1},\text{ where}
\end{equation}
\begin{equation}\label{e:polyu}
u^{2} - mu + 1 = 0, \; m \in \kk.
\end{equation}
\end{subequations}
Substituting $u$ for $r\sigma^{2}(r)$ in \eqref{e:M4}, solving for $\sigma(r)$, repeatedly applying $\sigma$, and keeping in mind Eq.\ \eqref{e:sigmau}, we obtain the expressions
\begin{equation}
\tag{C4} \label{e:C4}
\sigma(r) = \frac{-r - 1}{r + u}, \; \sigma^{2}(r) = \frac{u}{r}, \; \sigma^{3}(r) = \frac{-u^{-1}r - 1}{r + 1}, \; \sigma^{4}(r) = r.
\end{equation}
The product of the four expressions in \eqref{e:C4} is 1 as required by Proposition~\ref{p:norm1}.  The expression for $\sigma^{2}(r)$ is of the same form as that for a ``relative'' unit in a cyclic quartic extension, as given in \cite{wash:quartics}, \S 2.  These expressions are also strikingly similar to those in \cite{leut:euclenstra}, \S 1.1.

Because of Eq. \eqref{e:sigmau}, for the map $\sigma$ to fix $\kk$ elementwise it is necessary that either
\begin{equation}\label{e:notasquare}
m=\pm 2\text{ or } m^{2}-4\notin \kk^{2}.
\end{equation}
Substituting \eqref{e:C4} into the condition in Eq.~\eqref{e:admiss} with $d = 1$,
\begin{equation}
\tag{Tr4} \label{e:Tr4}
r + \frac{-r - 1}{r + u} + \frac{u}{r} + \frac{-u^{-1}r - 1}{r + 1} = A, \; A \in \kk.
\end{equation}
Clearing fractions and collecting terms, we find that $r$ is a zero of
\begin{equation}
\tag{Q4}\label{e:Q4}
p(x) = x^{4} + (u - u^{-1}-A)x^{3} + ((2-A)u-A - 4)x^{2} + (u^{2} - 1-Au)x + u^{2}
\end{equation}
where $m, A \in \kk$, and Eqs.~\eqref{e:polyu} and \eqref{e:notasquare} hold.  By Eq.~\eqref{e:notasquare}, $p(x)\in\kk[x]$ only if $m=\pm 2$.  We treat those cases first.
%
\subsection{The cases $m = \pm 2$}\label{ss:meqpm2}
If $m = 2$ then $u = 1$, the formal expressions for $\sigma(r)$ and $\sigma^{3}(r)$ become $-1$, and we obtain
$$p(x) = (x + 1)^{2}(x^{2} - (A + 2)x + 1).$$
In this case the transformation $\sigma$ is \emph{not} a field automorphism of order 4, but  if $(A + 2)^{2} - 4 \notin \kk^{2}$, the quadratic factor is irreducible in $\kk[x]$, and the expression $1/r$ for $\sigma^{2}(r)$ does define an automorphism of order 2.  The substitution $A \leftarrow -4 - A$
changes the signs of the zeroes of the quadratic factor.

If $m = -2$ then $u = -1$, and the expressions in \eqref{e:C4} become
\begin{equation}\tag{C4$'$}\label{e:oldc4}
\sigma(r)= \frac{-r - 1}{r - 1}, \,\sigma^{2}(r)=-\frac{1}{r}, \, \sigma^{3}(r)=\frac{r - 1}{r + 1}, \,\text{ and}\; \sigma^{4}(r)=r.
\end{equation}
Substituting $u = -1$ into Equation~\eqref{e:Q4} gives
\begin{equation}
\tag{Q4$'$}\label{e:simpq}
p(x) = x^{4} - Ax^{3} - 6x^{2} + Ax + 1.
\end{equation}
Apart from parameter name, this is  the $P(x)$ in Proposition \ref{p:known}(b) for the ``simplest'' quartic fields.  If $\kk = \Q$ and $A \in \Z$, $p(x)$ is irreducible in $\Z[x]$ unless $A \in \{-3, 0, 3\}$.  In general, $p(x)$ is irreducible in $\kk[x]$ when $A^{2} + 16 \notin \kk^{2}$.  The Galois group acts on the zeroes as in  \eqref{e:oldc4}.  The polynomials $f_{1}(x) = p(x)$ and $f_{3}(x)$ in Proposition~\ref{p:diffgen}, are related by the change of parameter $A \leftarrow -A$.
%
\subsection{Fundamental identities when $m^{2}-4\notin\kk^{2}$}
\label{ss:basicalgid}
When $m^{2}-4\notin\kk^{2}$, however, the $p(x)$ in \eqref{e:Q4} is \emph{not} in $\kk[x]$. Conjugating the coefficients in $\kk(u)/\kk$ gives a polynomial $\bar{p}(x)\ne p(x)$.  Assuming $\sigma$ is a field automorphism, and $p(r)=0$, $\sigma(r)$ is a zero of $\bar{p}(x)$ rather than of $p(x)$.  But \emph{that} would imply  an \emph{algebraic} relation between  $\bar{p}((-x-1)/(x+u))$ and $p(x)$.  And there is indeed such a relation:
\begin{thm}
\label{t:opensesame} Let $m, A \in \kk$, $m^{2}-4\notin \kk^{2}$, and $u$, $p(x)$ and $\bar{p}(x)$ be as above.  Then
$$(x + u)^{4}\;\bar{p}\left(\frac{-x - 1}{x + u}\right) = (m - 2)p(x).$$
\end{thm}
\begin{proof}Straightforward polynomial algebra.
\end{proof}
Although $p(x)\notin\kk[x]$, clearly
\begin{equation}\label{e:tmax}
T(m,A,x)=p(x)\bar{p}(x)
\end{equation}
is a monic octic in $\mathbb{Z}[m,A][x]$ (hence in $\kk[x]$ since $m,A\in\kk$), with  constant term $1$. Theorem \ref{t:opensesame} gives two important properties of $T(m,A,x)$:
\begin{cor}
\label{c:nrml}
Let $m$, $A$, and $p(x)$ be as above.
\begin{enumerate} \renewcommand{\labelenumi}{(\alph{enumi})}
\item $p(x)$ and $\bar{p}(x)$ have the same splitting field $L$ over $\kk(u)$, which is the splitting field of $T(m, A, x)$ over $\kk$.
\item $\textup{disc}(T(m,A,x))\in \kk^{2}$.
\end{enumerate}
\end{cor}
\begin{proof}Adjoining to $\kk(u)$ either the zero $r$ of $p(x)$, or the zero $(-r - 1)/(r + u)$ of $\bar{p}(x)$, define the \emph{same} extension of $\kk(u)$.  The join of the extensions defined by all the zeroes thus gives the same splitting field for both quartics.

For (b), a standard formula (see, for instance, \cite{cohen:compalgnothr}, Corollary 3.3.6) gives
$$\text{disc}(T(m,A,x))=\text{disc}(p(x))\text{disc}(\bar{p}(x))[\text{Res}(p(x),\bar{p}(x))]^{2}.$$
The arguments of the resultant are interchanged by conjugation in $\kk(u)/\kk$, but also by 4 (evenly many) row interchanges of the Sylvester's matrix, which leaves its determinant unchanged.  Therefore, the resultant is in $\kk$.

If we differentiate the identity in Theorem~\ref{t:opensesame}, evaluate at the zeroes of $p(x)$, and multiply, we obtain $\text{disc}(p(x))=u^{6}\text{disc}(\bar{p}(x))$.  The discriminants are conjugate in $\kk(u)/\kk$ and $u$ has norm $1$, so disc$(p(x)) =cu^{3}$ for some $c \in \kk$.  Thus, 
$\text{disc}(p(x))\text{disc}(\bar{p}(x))=c^{2}\in\kk^{2}$, and the proof is complete.
\end{proof}
The situation is further simplified by the fact that $p(x)$ is a generalized reciprocal polynomial:
\begin{lem}
\label{l:genrecip1} If $m,A \in \kk$, $u^{2}-mu+1=0$, and $p(x)$ is as in \eqref{e:Q4}, then
$$\frac{p(x)}{x^{2}} = \frac{p\left(u/x\right)}{(u/x)^{2}}.$$
\end{lem}
\begin{proof}
Straightforward polynomial algebra.
\end{proof}
We use Lemma~\ref{l:genrecip1} to split $p(x)$ into quadratic factors.
\begin{lem}
\label{l:genrecip2}With $p(x)$ as in \textup{Eq.~\eqref{e:Q4}}  and $Y = x + u/x$,
\begin{enumerate}
\renewcommand{\labelenumi}{(\alph{enumi})}
\item $p(x) = x^{2}F(Y)$, where $F(Y)=Y^{2} + (-A + u - u^{-1})Y - A - 4 - Au$.
\item $\textup{disc}(F(Y)) = (m + 2 + A)^{2} - 4(m - 2)$.
\item If $F(Y) = (Y - \alpha)(Y - \beta)$, then
$$p(x) = q_{1}(x)q_{2}(x)=(x^{2} - \alpha x + u)(x^{2} - \beta x + u).$$
\item  Let $Q_{1}=q_{1}(-1)=1+\alpha+u$, $Q_{2}=q_{2}(-1)=1 + \beta+u$.  Then
\begin{enumerate}\renewcommand{\labelenumii}{(\roman{enumii})}
\item $\textup{disc}(x^{2} - \alpha x + u)=\alpha^{2}-4u=Q_{1}(Q_{1}+Q_{2} u-2(u+1))$.
\item $\textup{disc}(x^{2} - \beta x + u)=\beta^{2}-4u=Q_{2}(Q_{2}+Q_{1} u -2(u+1))$.
\end{enumerate}
\end{enumerate}
\end{lem}
\begin{proof}
Straightforward polynomial algebra.
\end{proof}
We then have
\begin{prop}\label{p:2pow}
Let $m,A\in\kk$, $m^2-4\notin\kk^{2}$, and $L/\kk$ the splitting field of $T(m,A,x)$.  Then $[L\!:\!\kk]$ divides \textup{16}.
\end{prop}
\begin{proof}
Since $p(x)$ splits into quadratic factors in a quadratic extension (at most)  of its coefficient field $\kk(u)$, its splitting field over $\kk(u)$ has degree dividing $8$, and $\kk(u)$ has degree $2$ over $\kk$.  Corollary \ref{c:nrml}(a) then gives the result.
\end{proof}
%
\subsection{The quadratic and quartic factors of $T(m,A,x)$}\label{ss:quadfacts}
Note that \\ disc$(F(Y))\in\kk$ in Lemma \ref{l:genrecip2}(b), even though $F(Y)\notin\kk[Y]$.
The quadratic factors $q_{1}(x)$ and $q_{2}(x)$ of $p(x)$ in Lemma \ref{l:genrecip2}(c)  (and, obviously, the corresponding factors of $\bar{p}(x)$) are in $\kk(s,w)[x]$, where
\begin{equation}\label{e:newvars}
s^{2}=m^{2}-4\text{ and }w^{2}=(m+2+A)^{2} -4(m-2)=\text{disc}(F(Y)).
\end{equation}
Taking $u=(m+s)/2$ and solving $F(Y)=0$ by quadratic formula, we find the quadratic factors of $T(m,A,x)$ in $\kk(s,w)[x]$ are
\begin{subequations}
\begin{align}
\label{e:q1}&q_{1}(x)=x^{2}+(-A+s-w)x/2 +(m+s)/2\\
\label{e:q2}&q_{2}(x)=x^{2}+(-A+s+w)x/2 +(m+s)/2,\\
\label{e:q3}&q_{3}(x)=x^{2}+(-A-s+w)x/2 +(m-s)/2\text{ , and}\\
\label{e:q4}&q_{4}(x)=x^{2}+(-A-s-w)x/2 +(m-s)/2.
\end{align}
\end{subequations}
These choices give  $q_{1}(x)q_{2}(x)=p(x)$ and $q_{3}(x)q_{4}(x)=\bar{p}(x)$.
By regrouping the $q_{i}(x)$ in pairs, we see that $T(m,A,x)$ splits into quartic factors in $\kk(s)[x]$, $\kk(sw)[x]$, and $\kk(w)[x]$, namely
\begin{subequations}
\begin{align}
\label{e:Ps}&P_{s}(x)=p(x)=q_{1}(x)q_{2}(x)\text{ and }\bar{P}_{s}(x)=\bar{p}(x)=q_{3}(x)q_{4}(x),\\
\label{e:Psw}&P_{sw}(x)=q_{1}(x)q_{3}(x)\text{ and }\bar{P}_{sw}(x)=q_{2}(x)q_{4}(x),\text{ and}\\
\label{e:Pw}&P_{w}(x)=q_{1}(x)q_{4}(x)\text{ and }\bar{P}_{w}(x)=q_{2}(x)q_{3}(x).
\end{align}
\end{subequations}
We have the following refinement of Theorem \ref{t:opensesame}:
\begin{lem}\label{l:identqs}
Let $m,A\in\kk$, $m^2-4\notin\kk^{2}$, and notation as above.  Then
\begin{align*}
&\textup{a) }(x+u)^{2}\,q_{3}\left(\frac{-x-1}{x+u}\right)=q_{3}(-1)q_{1}(x);\\
&\textup{b) }\textup{disc}(q_{1}(x))\textup{disc}(q_{3}(x))=((u-1)\textup{disc}(q_{3}(x))/q_{3}(-1))^{2};\\
&\textup{c) }(u-1)\textup{disc}(q_{3}(x))/q_{3}(-1)=((m + A - 2)s + (2-m)w)/2.
\end{align*}
\end{lem}
\begin{proof}For (a), there clearly must be an algebraic relation between $q_{1}(x)$ and either $q_{3}((-x-1)/(x+u))$ or $q_{4}((-x-1)/(x+u))$.  By Eqs.~\eqref{e:q1} and \eqref{e:q3}, the coefficient of $x^{3}$ in $q_{1}(x)q_{3}(x)$ is $-A$ in agreement with \eqref{e:Tr4}.  Routine though tedious algebra completes the verification.

For (b), differentiate (a), evaluate at the zeroes, and multiply.

Finally, (c) may be verified algebraically.
\end{proof}
\begin{rems}
The substitution $w\leftarrow -w$  replaces $(q_{1}(x), q_{3}(x))$ with $(q_{2}(x),q_{4}(x))$ in (a) -- (c).  Multiplying (a) by its counterpart gives Theorem \ref{t:opensesame}.
\end{rems}
To formulate $s$ (thus $u=(m+s)/2$ and $\sigma$) in the arithmetic of $\kk[x]$, let
\begin{subequations}
\begin{align}
\label{e:srat1}&P_{s}(x)=p(x)=\mathcal{A}(x)s+\mathcal{B}(x)\text{ where }\mathcal{A}(x),\mathcal{B}(x)\in\kk[x],\text{ and take}\\
\label{e:srat2}&s\equiv-\mathcal{B}(x)/\mathcal{A}(x)\pmod{T(m,A,x)},\text{ if Res}(\mathcal{A}(x),T(m,A,x))\ne 0.
\end{align}
\end{subequations}
Similarly, writing
\begin{subequations}
\begin{align}
\label{e:wrat1}&P_{w}(x)=\mathcal{C}(x)w+\mathcal{D}(x)\text{ where }\mathcal{C}(x),\mathcal{D}(x)\in\kk[x], \text{ we can take}\\
\label{e:wrat2}&w\equiv-\mathcal{D}(x)/\mathcal{C}(x)\pmod{T(m,A,x)} \text{ if Res}(\mathcal{C}(x),T(m,A,x))\ne 0.
\end{align}
\end{subequations}
We can express these resultants in terms of the ``monster'' resultant norm
\begin{equation}\label{e:monster}
\mu=\text{Res}(q_{1}(x),q_{3}(x))\text{Res}(q_{2}(x),q_{4}(x)).
\end{equation}
 Direct calculation (using plenty of computing power)  gives the results
\begin{subequations}
\begin{align}
\label{e:monsters}&\text{Res}(\mathcal{A}(x),T(m,A,x))=(m-2)^{2}\mu^{2}/256,\text{  and}\\
\label{e:monsterw}&\text{Res}(\mathcal{C}(x),T(m,A,x))=\mu^{2}/256.
\end{align}
\end{subequations}
\begin{lem}\label{l:murphy}
Let $m,A \in \kk$, and either $m=\pm 2$ or $m^2 - 4 \notin\kk^2$, with the $q_{i}(x)$ as in \textup{Eqs. \eqref{e:q1}--\eqref{e:q4}} and $\mu$ as in \textup{Eq. \eqref{e:monster}}.  Then,
$$\mu\ne 0\text{ unless}$$
$(m,A) = (2,-4)$; $(-2, 4i)$ or $(-2, -4i)$ if $-1\in\kk^{2}$; or $(2/3, -4/3)$ if $-2\notin\kk^{2}$.
\end{lem}
\begin{proof}
If $\text{Res}(q_{1}(x),q_{3}(x)) = 0$, then $q_{1}(x)$ and $q_{3}(x)$ have a common factor in $\kk(s,w)[x]$, which is also a factor of every $\kk(s,w)$-linear combination of $q_{1}(x)$ and $q_{3}(x)$.  Then, by Lemma \ref{l:identqs}(a),  $q_{u}(x) = x^2 + (1 + u)x + 1$ or $\bar{q}_{u}(x) = x^2 + (u^{-1} + 1)x + 1$ has a common factor of degree $\ge 1$ with $q_{1}(x)$ and $q_{3}(x)$, because any common zero of $q_{1}(x)$ and $q_{3}(x)$ is a fixed point of (at least) one of the linear fractional transformations $\sigma(x)$ or $\sigma^{3}(x)$ in Eq. \eqref{e:C4}.

There can be no common factor of degree $\ge 1$ unless $\{q_{1}(x), q_{u}(x), q_{3}(x)\}$ or $\{q_{1}(x), \bar{q}_{u}(x), q_{3}(x)\}$ is a linearly dependent set in the $\kk(s,w)$ vector space $V$ with basis $\{1,x,x^{2}\}$.  Let $v=(q_{1}(x), q_{u}(x), q_{3}(x)) \in V^3$, and let  $M$ be the $3\times 3$ matrix whose $i,j$ entry is the coefficient of  $x^{3-j}$ in $v_{i}$.  The entries of $v$ are linearly dependent when the ``test value'' $\mathbf{tv}=\text{Det}(M)$ is $0$.  We have

$$\mathbf{tv}=(-2m - A)s/2+(m - 2)w/2 - (m^2 - 4)/2.$$

If $m = 2$, all three terms are $0$.  Then $w^2 = (A+4)^2$.  Taking $w = A+4$ and $s = 0$, $q_{1}(x) =x^2 - (A+2)x + 1$ and $q_{3}(x)=(x+1)^2$ as in \S\ref{ss:meqpm2}. Then $\gcd((q_{1}(x),q_{3}(x))=1$  unless $A = -4$, when $q_{1}(x)\equiv q_{3}(x)=(x+1)^2$.

If $m = -2$, $s=0$ and $\mathbf{tv}=0$ when $w = m+2=(-2)+2=0$.  But also $w^2 = A^2 + 16$, which is $0$ only for $A=\pm 4i$, when $q_{1}(x)\equiv q_{3}(x)=(x\mp i)^2$.

If $m\ne\pm 2$, then $s \notin \kk$. If $(-2m-A)\ne 0$, $(-2m-A)s/2\notin\kk$.  Then $\mathbf{tv}\notin\kk$ unless $(-2m-A)s/2+(m-2)w/2 = 0$.  But then $\mathbf{tv}\ne 0$ since $m\ne\pm 2$.

So if $m\ne \pm 2$, $\mathbf{tv}\ne 0$ unless $A = -2m$.  Again $w=m+2$, but also $w^2 = (m + 2 + (-2m))^2 - 4(m-2) = m^2 - 8m + 12$.  The two conditions on $w$ are only satisfied simultaneously when $m = 2/3$ and $A = -4/3$.  Substituting $(s,w)\leftarrow(-s,-w)$ changes $v$ to $v'=(q_{3}(x),\bar{q}_{u}(x), q_{1}(x))$, whose ``test value'' $\mathbf{tv}'$ is $0$ for the same $(m,A)$.  Substituting $(s,w)\leftarrow (s,-w)$ gives the argument for $\text{Res}(q_{2}(x),q_{4}(x))$.  If $m=2/3$, $m^{2}-4\notin\kk^{2}$ when $-2\notin\kk^{2}$.
\end{proof}
\begin{rem}
Direct calculation shows that
$$\mathbf{tv}\cdot\mathbf{tv}'=(2-m)\text{Res}(q_{1}(x),q_{3}(x)).$$
\end{rem}
\begin{lem}\label{l:contsw}
Let $m,A\in\kk$, $m^{2}-4\notin\kk^{2}$, and $s$ and $w$ as in \textup{Eq. \eqref{e:newvars}}.  If $T(m,A,r)=0$, then $\kk(s,w)\subset \kk(r)$.
\end{lem}
\begin{proof}
We may take $r\equiv x\pmod{T(m,A,x)}$.  Since $m\ne\pm 2$,  $s$ and $w$ are given by Eqs. \eqref{e:srat2} and \eqref{e:wrat2} unless $(m,A)=(2/3,-4/3)$, by Lemma \ref{l:murphy}.  This case may be verified directly.
\end{proof}
\begin{prop}\label{p:wnrml}
Let $m,A\in \kk$, $m^{2}-4\notin\kk^{2}$.  Then $P_{w}(x)$ and $\bar{P}_{w}(x)$ have the same splitting field over $\kk(w)$.
\end{prop}
\begin{proof}
This follows from an argument similar to that for Corollary \ref{c:nrml}(a), using Lemma \ref{l:contsw}.
\end{proof}
Brute-force algebra gives the simple expressions 
\begin{subequations}
\begin{align}
\label{e:dp12}&\text{disc}(q_{1}(x))\text{disc}(q_{2}(x))=(A^{2}-4(m-2))\cdot (u-1)^{2},\\
\label{e:res12}&\text{Res}(q_{1}(x),q_{2}(x))=w^{2}u,\\
\label{e:dp14}&\text{disc}(q_{1}(x))\text{disc}(q_{4}(x))=(A^{2}-4(m-2))\cdot Q_{1}^{2},\text{ and}\\
\label{e:res14}&\text{Res}(q_{1}(x),q_{4}(x))=s^{2}Q_{1}=(m^{2}-4)Q_{1},
\end{align}
\end{subequations}
where $Q_{1}=q_{1}(-1)=(m+2+A+w)/2$ is as in Lemma~\ref{l:genrecip2}(d).

From Eqs.~\eqref{e:dp12} and \eqref{e:dp14}, we see that the splitting field of $T(m,A,x)$ contains $\kk(y)$, where
\begin{equation}\label{e:lastnewvar}
y^{2}=A^{2}-4(m-2).
\end{equation}
We have the following result:
\begin{prop}\label{p:simple}
Let $m,A\in\kk$, $m^{2}-4\notin\kk^{2}$.  The zeroes of $T(m,A,x)$ are simple unless $w^{2}y^{2}=0$ or $(m,A)=(2/3,-4/3)$.
\end{prop}
\begin{proof}
Eqs. \eqref{e:dp12} -- \eqref{e:res14} show that if $m\ne\pm 2$, $w^{2}\ne 0$, and $y^{2}\ne 0$, then none of the $q_{i}(x)$ have repeated factors, and $\text{Res}(q_{i}(x),q_{j}(x))\ne 0$ unless $\{i,j\}=\{1,3\}$ or $\{2,4\}$.  For these, apply Lemma \ref{l:murphy}.
\end{proof}
When $w=0$, $q_{1}(x)\equiv q_{2}(x)$ and $q_{3}(x)\equiv q_{4}(x)$, so $P_{w}(x)\equiv P_{sw}(x)$ and $T(m,A,x)=P^{2}_{sw}$ in $\kk[x]$. We deal with this case in \S\ref{ss:washquart}.  We see from Lemma \ref{l:identqs}(a) and its counterpart for $q_{2}(x)$ and $q_{4}(x)$, that except for the four pairs $(m,A)$ in Lemma \ref{l:murphy}, if $y^{2}\ne 0$ the zeroes of $P_{sw}(x)$ and $\bar{P}_{sw}(x)$ are simple, and the linear fractional transformation $\sigma$ in Eq. \eqref{e:C4} (defined via Eq. \eqref{e:srat2}) permutes the zeroes of each cyclically.

The expression in Lemma \ref{l:identqs}(c) for a square root of the discriminant norm leads to an irreducibility criterion for  $P_{sw}(x)$ and $\bar{P}_{sw}(x)$ in $\kk(sw)[x]$.
\newpage
\begin{thm}\label{t:irrcyc}Let $m,A\in\kk$, $m^2-4\notin\kk^2$.  Then
\begin{enumerate}\renewcommand{\labelenumi}{(\alph{enumi})}
\item $P_{sw}(x)$ and $\bar{P}_{sw}(x)$ are both irreducible in $\kk(sw)[x]$ if and only if
$$[\kk(s,w)\!:\!\kk(sw)]=2\;\text{and}\,\;A^{2}-4(m-2)\ne 0.$$
\item If $P_{sw}(x)$ and $\bar{P}_{sw}(x)$ are irreducible in $\kk(sw)[x]$, they define cyclic quartic extensions of $\kk(sw)$.  In each case, the action of the Galois group on the zeroes is given by \eqref{e:C4}, defined via \textup{Eq. \eqref{e:srat2}}.
\end{enumerate}
\end{thm}
\begin{proof} First, $P_{sw}(x)=q_{1}(x)q_{3}(x)$ in $\kk(s,w)[x]$, so is reducible in $\kk(sw)[x]$ if $\kk(sw)=\kk(s,w)$. So suppose $[\kk(s,w)\!:\!\kk(sw)]=2$.  Conjugation in $\kk(s,w)/\kk(sw)$ changes the sign of the square root of  $\text{disc}(q_{1}(x))\text{disc}(q_{3}(x))$ in Lemma~\ref{l:identqs}(c), so neither $\text{disc}(q_{1}(x))$ nor $\text{disc}(q_{3}(x))$ is a square in $\kk(s,w)$ unless their product is 0.  Similarly for $q_{2}(x)$, $q_{4}(x)$ and $\bar{P}_{sw}(x)$.  Multiplying these square roots gives $(m-2)(A^{2}-4(m-2))$ as a square root of $\text{disc}(q_{1}(x))\text{disc}(q_{2}(x))\text{disc}(q_{3}(x))\text{disc}(q_{4}(x))$.  Now $m\ne 2$ by hypothesis, so if $A^{2}-4(m-2)\ne 0$, both $P_{sw}(x)$ and $\bar{P}_{sw}(x)$ are irreducible in $\kk(sw)[x]$.  But if $A^{2}-4(m-2)= 0$, at least one of $P_{sw}(x)$ and $\bar{P}_{sw}(x)$ has a repeated factor, so is reducible in $\kk(sw)[x]$.

Part (b) follows from the discussion after Eqs. \eqref{e:dp12}--\eqref{e:res14}.
\end{proof}
We then obtain an irreducibility criterion for $T(m,A,x)$:
\begin{cor}\label{c:irroct}
Let $m,A\in\kk$.  Then $T(m,A,x)$ is irreducible in $\kk[x]$  if and only if $s^2\notin\kk^2$, $w^2\notin\kk^2$, and $s^{2}w^2\notin\kk^2$, i.e. if and only if $[\kk(s,w)\!:\!\kk]=4$.
\end{cor}
\begin{proof}
Clearly, $T(m,A,x)$ is irreducible in $\kk[x]$ if and only if  $P_{sw}(x)$ and $\bar{P}_{sw}(x)$ are irreducible in $\kk(sw)[x]$ (which requires $[\kk(s,w)\!:\!\kk(sw)]=2$), and $[\kk(sw)\!:\!\kk]=2$.  We only need to show that, if $A^{2}-4(m-2)=0$, then $[\kk(s,w)\!:\!\kk]<4$.  This is trivial if $m=2$.  Otherwise, $A\ne 0$, so we can write
$$w^{2}=(m+2)(m+2+2A)=(2s/A)^2(A/2+2)^2\in(\kk(s))^{2}.$$
\end{proof}
%
\subsection{Pairs of related octics}\label{ss:relocts}
Eqs.~\eqref{e:newvars} and \eqref{e:dp12}-\eqref{e:res14} show the splitting field of $T(m,A,x)$ contains $E=\kk(s,w,y)$, where
$$s^{2}=m^{2}-4,\;w^{2}=(m+A+2)^{2}-4(m-2),\text{ and }y^{2}=A^{2}-4(m-2).$$
Note that the substitution $A\leftarrow -m-2-A$ has compositional order 2, and interchanges $w^{2}$ and $y^{2}$.  [This substitution gives the changes of parameter in \S\ref{ss:meqpm2}.] We call $T(m,A,x)$ and $T(m,-m-2-A,x)$ \emph{related} octics.  The following result indicates how closely they are related:
\begin{thm}\label{t:relocts}Let $m,A\in\kk$, and $s^{2}=m^{2}-4\notin\kk^{2}$.  Then $T(m,A,x)$ and $T(m,-m-2-A,x)$ have the same splitting field over $\kk$.
\end{thm}
\begin{proof}
We consider the corresponding quadratic factors
\begin{align*}
&q_{1}(x) = x^{2}+(-A+s-w)x/2 +(m+s)/2 \text{ as in Eq.\ \eqref{e:q1}, and}\\
&\psi_{1}(x)=x^{2}+(m+2+A+s-y)x/2 +(m+s)/2 \text{ with $y$ as in Eq.\ \eqref{e:lastnewvar}}.
\end{align*}
If $w^{2}y^{2}\ne 0$, Eqs.~\eqref{e:dp12}, \eqref{e:dp14}, and Lemma \ref{l:identqs} show that the $q_{i}(x)$ all define the \emph{same} extension of $E$.  Similarly for the corresponding $\psi_{i}(x)$.  Greatly facilitated by the substitution $A=-(m+2)/2+\mathbf{v}$ and Phil Carmody's guidance with Pari-GP, we found an algebraic square root of $\text{disc}(q_{1}(x))\text{disc}(\psi_{1}(x))$ in $E$, namely
$$c_{0}+c_{1}s+c_{2}w+c_{3}y+c_{4}sw+c_{5}sy+c_{6}wy+c_{7}swy,\text{ where}$$
\begin{equation*}
\begin{split}
&c_{0}=(3m^2 - 20m + 12+4\mathbf{v}^2)/16 ,\;c_{1}=(m - 6)/4,\;c_{2}=(-3m + 2+2\mathbf{v})/8,\\
&c_{3}=(-3m + 2-2\mathbf{v})/8,\;c_{4}=-1/4,\;c_{5}=-1/4,\;c_{6}=-1/4,\;\text{and}\,\;c_{7}=0.
\end{split}
\end{equation*}
If $w^{2}=0$, at least one of $\text{disc}(q_{1}(x))\text{disc}(\psi_{1}(x))$ and $\text{disc}(q_{1}(x))\text{disc}(\psi_{2}(x))$ will be nonzero unless $y^{2}=0$ also.  (Replacing $\psi_{1}(x)$ with $\psi_{2}(x)$ has the effect of replacing $y$ with $-y$ in the algebraic square root.)  But $w^{2}=y^{2}=0$ only when $m=6$ and $A=-4$, for which the two related octics are identical, $T(6,-4,x)=(x^{2}+2x-1)^{4}$.
\end{proof}
%
\subsection{The Galois group of $T(m,A,x)$}\label{ss:galgrp}
We can now prove the main results about the splitting field $L/\kk$ of $T(m,A,x)$.
\begin{thm}\label{t:galgp1}Let $m,A\in\kk$, and $[E\!:\!\kk]=8$ where $E=\kk(s,w,y)$.  Then
\begin{enumerate}\renewcommand{\labelenumi}{(\alph{enumi})}
\item $T(m,A,x)$ is irreducible in $\kk[x]$ with Galois group $G\cong$ $_{8}T_{11}$.
\item The fixed field of the quaternion subgroup is $\kk(swy)$.
\item The fixed fields of the subgroups $\cong D_{4}$ are $\kk(s)$, $\kk(w)$, and $\kk(y)$.
\item The fixed fields of the subgroups $\cong C_{4}\times C_{2}$ are $\kk(sw)$, $\kk(sy)$, and $\kk(wy)$.
\end{enumerate}
\end{thm}
\begin{proof}
$T(m,A,x)$ is irreducible in $\kk[x]$ by Corollary~\ref{c:irroct}.  If $L/\kk$ is its splitting field, then $E\subset L$, $[L\!:\!\kk]\mid 16$ by Proposition~\ref{p:2pow}, and $L$ contains a $C_{4}$ extension of $\kk(sw)$ by Theorem \ref{t:irrcyc}, so $L\ne E$.  Thus $\abs{G}=16$, and $G\subset A_{8}$ by Corollary~\ref{c:nrml}(b).  This narrows the possibilities for $G$ to $_{8}T_{9}$, $_{8}T_{10}$, and $_{8}T_{11}$ .  Of these, only $_{8}T_{11}$ has a quaternion subgroup.

Let $H$ be the subgroup of $G$ fixing $\kk(swy)$.  Then $\abs{H}=8$.  Since the $q_{i}(x)$  remain irreducible in $E[x]$, it follows that $P_{s}(x)$, $P_{sw}(x)$, and $P_{w}(x)$ remain irreducible in $\kk(s,swy)[x]$, $\kk(sw,swy)[x]$, and $\kk(w,swy)[x]$, respectively.  By Lemma \ref{l:identqs} and Eqns.~\eqref{e:dp12} and \eqref{e:dp14}, their discriminants are not squares in the coefficient fields, and $L$ is normal of degree $4$ over each, so
$$G(L/\kk(s,swy))\cong G(L/\kk(sw,swy))\cong G(L/\kk(w,swy))\cong C_{4}.$$
Now the elements of order 4 in these groups can only map $q_{1}(x)$ to $q_{2}(x)$, $q_{3}(x)$, and $q_{4}(x)$, respectively.  Thus, $H$ has $6$ elements of order 4, so $H\cong Q_{8}$, proving (a) and (b).  By Corollary \ref{c:nrml},  $G(L/\kk(s))\cong D_{4}$.   By Proposition \ref{p:wnrml}, $G(L/\kk(w))\cong D_{4}$ also; and $G(L/\kk(y))\cong D_{4}$ (use the related octic), establishing (c).  The other three maximal subgroups of $_{8}T_{11}$ are $\cong C_{4}\times C_{2}$, so (d) follows by the Galois correspondence.
\end{proof}
\begin{rems}
The degree of $E/\kk$ can be determined by testing whether $s^2$, $w^2$, $y^2$, $s^2w^2$, $s^2y^2$, $w^2y^2$, and $s^2w^2y^2$ are squares in $\kk$.

When $[E\!:\!\kk]=8$, the factors $P_{sw}(x)$ and $\bar{P}_{sw}(x)$ of $T(m,A,x)$ define the fixed fields of a conjugacy class of order-2 subgroups of $G$.  The corresponding factors of the related octic  define the fixed fields of a \emph{different} conjugacy class of order-2 subgroups.

Taking $\sqrt{a}=sw$, $\sqrt{b}=sy$, $\sqrt{c}=s$, and $d=\text{disc}(q_{1}(x))$ in \cite{minsmith:cfields}, Appendix, we have $\text{disc}(q_{1}(x))\text{disc}(q_{2}(x))=\text{k}_{a}$, $\text{disc}(q_{1}(x))\text{disc}(q_{3}(x))=\text{k}_{c}$, and $\text{disc}(q_{1}(x))\text{disc}(q_{4}(x))=\text{k}_{ac}$.  Using Eqs.\  \eqref{e:dp12}, \eqref{e:dp14}, and Lemma \ref{l:identqs}, the obstruction $(a,b)(c,c)$ to Galois group $\mathbf{DC}\cong$ $_{8}T_{11}$ then evaluates to $1$.  
\end{rems}
\begin{cor}\label{c:galgp2}
Let $m,A\in\kk$, $m^{2}-4\notin\kk^{2}$, $[E\!:\!\kk]=4$ where $E=\kk(s,w,y)$, and  $L/\kk$ the splitting field of $T(m,A,x)$.  Then $[L\!:\!\kk]=8$, and 
$$G(L/\kk) \cong \begin{cases}
Q_{8},\;\textup{if}\; swy\in\kk,\\
D_{4}\cong D_{8}(8),\;\textup{if}\; w\in\kk\;\textup{or}\; y\in\kk,\text{ and}\\
C_{4}\times C_{2},\;\textup{if}\;sw\in\kk,\;sy\in\kk,\;\textup{or}\; wy\in\kk.
\end{cases}$$
\end{cor}
\begin{proof}
At least one of the related octics is irreducible in $\kk[x]$ by Corollary~\ref{c:irroct}, and the rest is clear from Theorem~\ref{t:galgp1}.
\end{proof}
We formulate the condition $w\in\kk$ (disregarding $[E\!:\!\kk]$) by taking
\begin{equation}\label{e:dihedparms}
m\in\kk,\;d\in\kk^{\times},\;A=-m-2-d-(m-2)/d,\text{ and }w=(m-2)/d-d.
\end{equation}
Substituting into Eq.~\eqref{e:Pw}, we obtain
\begin{equation}\label{e:dihedquartic}
P_{w}(x)=x^{4} + (m\! +\! 2d\!+\! 2)x^{3} + ((d\!+\! 2)m\!+\! d^{2}\!+\!2d\!+\!2)x^2 + ((d\!+\!1)m\!+\! 2)x + 1.
\end{equation}
If $[E\!:\!\kk]=4$, the related octic has Galois group $D_{8}(8)$, and is a defining polynomial for the splitting field $L$, while $P_{w}(x)$ as in Eq.~\eqref{e:dihedquartic} and $\bar{P}_{w}(x)$ are quartics in $\kk[x]$ with $G\cong D_{4}$ which define the same conjugacy class of (non-normal) quartic intermediate fields of $L/\kk$.   Note that replacing $d$ by $(m-2)/d$ interchanges $P_{w}(x)$ and $\bar{P}_{w}(x)$.  If $d=(m-2)/d = t$, then $P_{w}(x)=\bar{P}_{w}(x)=P_{t}(x)$ as in \S\ref{ss:washquart}.  We give two other ``degenerate" cases of Eqs.~\eqref{e:dihedparms} and \eqref{e:dihedquartic}.  In both cases, the ``generic'' Galois group is $V_{4}$.

When $d=-2$, $A=-m-2-A$, so the related octics are identical.  (This happens for the pair $(m,A)=(2/3,-4/3)$ in Lemma \ref{l:murphy}).  In this case, $\text{disc}(P_{w}(x))\in\kk^{2}$.  We find $L=\kk(\sqrt{(m-2)^{2}-16},\sqrt{(m-4)^{2}-4})$.

If $d=-1$ then $A=-3$, and $\text{disc}(P_{w}(x))\in\kk^{2}$ when $m=4-t-t^{2}$, $t\in\kk$.  In this case, $L=\kk(\sqrt{t^{2}-4},\sqrt{(t+1)^{2}-4})$.

If $swy\in\kk$ or $wy\in\kk$ and $[E\!:\!\kk]=4$, the related octics are distinct defining polynomials for the splitting field.  We do not have a complete description of either $swy\in\kk$ or $wy\in\kk$.   With respect to $m$, Res$(w^{2},y^{2})=(A+4)^2(A^2+16)$.  Taking $A=-4$, we find:
\begin{prop}\label{p:c4c2andquatrn}
Let $t\in\kk$.
\begin{enumerate}\renewcommand{\labelenumi}{(\alph{enumi})}
\item If $t^2+4\notin\kk^2$, $t^2-4\notin\kk^2$, and $t^4-16\notin\kk^2$, then the related octics $T(2-t^{2},-4,x)$ and $T(2-t^{2}, t^{2},x)$ are irreducible in $\kk[x]$.  Both have the same splitting field $L$ with $G(L/\kk)\cong C_{4}\times C_{2}$.
\item If $t^2+4\notin\kk^2$, $t^2+8\notin\kk^2$, and $(t^2+6)^2-4\notin\kk^2$, then the related octics $T(-t^{2}-2,-4,x)$ and $T(-t^{2}-2, t^{2}+4,x)$ are irreducible in $\kk[x]$.  Both have the same splitting field $L$ with $G(L/\kk)\cong Q_{8}$.
\end{enumerate}
\end{prop}
\begin{proof}
Check that $[E\!:\!\kk]=4$; then $wy\in\kk$ in (a), and  $swy\in\kk$ in  (b).
\end{proof}
\begin{rems}
If $(m,A)=(-t^2+2- 8i,4i)$ (or $(-t^2+2+ 8i,-4i)$), $t\in\Z[i]$, then $T(m,A,x)$ and the related octic are ``typically'' distinct defining polynomials for a $C_{4}\times C_{2}$ extension of $\Q(i)$ ($wy\in\kk$).

With respect to $A$, $\text{Res}(w^{2},y^{2})=(m-2)^{2}(m-6)^{2}$.  Taking $m=6$, we find that if $A=8a_{n}-4$ where  $a_{n}+b_{n}\sqrt{2}=(3+2\sqrt{2})^{n}$ and $n\in\Z^{+}$, then $T(6,A,x)$ and $T(6, -8-A,x)$ are defining polynomials for a quaternion field in a family reminiscent of the cyclic octic fields in \cite{shen:uniclassoct}.
\end{rems}
If $sw\in\kk$ and $[\kk(s,y)\!:\!\kk]=4$, $P_{sw}(x)$ and $\bar{P}_{sw}(x)$ are Murphy's twins in $\kk[x]$, and define distinct cyclic quartic extensions of $\kk$.  The related octic is then a defining polynomial for their join.  We deal with Murphy's twins for $\kk=\Q$ and $m,A\in\Z$,  in \S\ref{ss:murphtwins}.
%
\subsection{Washington's cyclic quartic fields}\label{ss:washquart}
 We close this section with a discussion of the ``degenerate'' case $w^{2}=0$.  It produces the polynomials $P(x)$ in Proposition \ref{p:known}(c).  These are alternate defining polynomials for the cyclic quartic fields in \cite{wash:quartics} when $\kk=\mathbb{Q}$ and $t\in\mathbb{Z}-\{0,-2\}$.
 
Very early on, Phil Carmody drew the author's attention to examples of $T(m,A,x)$, $m,A\in\Z$, $m\ne\pm 2$, with repeated factors in $\Z[x]$.  It was this observation that originally led us to consider the case $w^{2}=0$.  As when $s^{2}=0$, $T(m,A,x)$ is again the square of a quartic in $\kk[x]$.  

If $m=t^2+2$, then $4(m-2) = 4t^{2}$, and $w=0$ when $A=-t^{2}-4\pm 2t$.  To choose the $\pm$ sign, we substitute into the formulas in Lemma \ref{l:genrecip2}(d); we find that $Q_{1}=Q_{2}=\pm t$, and 
\begin{align}
&\label{e:dwz}\text{disc}(q_{1}(x))=\text{disc}(q_{2}(x))=t(t\mp 2)(u+1), \text{ whence}\\
&\label{e:wzdiscnorm}\text{disc}(q_{1}(x))\text{disc}(q_{3}(x))=t^{2}(t \mp 2)^2(t^{2}+4).
\end{align}
Taking $A=-t^{2}-4-2t$ gives $\text{disc}(q_{1}(x))\text{disc}(q_{3}(x))=t^{2}(t+2)^{2}(t^{2}+4)$, the same form as the product of the discriminants of the quadratic factors of the quartics in  \cite{wash:quartics}.  Then $T(t^{2}+2,-t^{2}-2t-4,x)=P^{2}_{t}(x)$, where
\begin{equation}\label{e:hsquart}
P_{t}(x)=x^4 + (t^2 + 2t + 4)x^3 + (t^3 + 3t^2 + 4t + 6)x^2 + (t^3 + t^2 + 2t + 4)x + 1,
\end{equation}
the $P(x)$ in Proposition \ref{p:known}(c).  By Theorem \ref{t:irrcyc}, $P_{t}(x)$ is irreducible in $\kk[x]$ when $t\in\kk$, $t^{2}+4\notin\kk^{2}$, and $t\ne -2$.

Next, we relate the $P_{t}(x)$ to the cyclic quartics
\begin{equation}\label{e:washquart}f_{t}(x) = x^{4}-t^{2}x^{3}-(t^{3}+2t^{2}+4t+2)x^{2}-t^{2}x+1\end{equation}
in \cite{wash:quartics}. Clearly $f_{t}(x)$ is a reciprocal polynomial; if $f_{t}(\rho)=0$ then the element of order 2 in the Galois group maps $\rho$ to $1/\rho$.  Similarly, the element $\sigma^{2}$ of order 2 in the Galois group of $P_{t}(x)$ maps a zero $r$ to $u/r$.  With $m=t^{2}+2$, we have $u^{2}-(t^2+2)u+1= 0$, and $u = v^{2}$ where $v^{2}-tv-1= 0$.  Then $\sigma^{2}(r/v)=1/(r/v)$, so $r/v$ is a zero of a reciprocal polynomial.  By formulating $u$ and $\sigma$ as rational expressions$\pmod{P_{t}(x)}$, and using Pari-GP to bludgeon the algebra into submission, we find that
\begin{equation}\label{e:ptu}
u\equiv -x^{3} - (t^{2} + 2t +3)x^{2} - (t^{3} + 2t^{2} + 3t + 3)x - 1\pmod{P_{t}(x)}.
\end{equation}
Taking $v=(u-1)/t$, we find that $r/v$ is in fact a zero of $f_{t}(x)$.
\begin{prop}\label{p:ptwash}If $t\in\kk- \{0,-2\}$, $P_{t}(x)$ and $f_{t}(x)$ define the same extension of $\kk$.
\end{prop}
\begin{proof}
Eq.\ \eqref{e:ptu} and $v=(u-1)/t$ give $f_{t}(x/v)\equiv 0\pmod{P_{t}(x)}$, where
$$x/v\equiv (x^3 + (t^2 + t + 3)x^2 + (t^2 + t + 3)x + 1)/t^2\pmod{P_{t}(x)},\text{ if }t\ne 0.$$
Reformulating $v\pmod{f_{t}(x)}$, we have $P_{t}(xv)\equiv 0\pmod{f_{t}(x)}$, where
$$xv\equiv (-x^2 + (t^2 + t)x - 1)/(t + 2)\pmod{f_{t}(x)},\text{ if }t\ne -2.$$
Both transformations are defined when $t\in\kk-\{0,-2\}$, and the result follows.
\end{proof}
\begin{rems}
The related octic $T(t^{2}+2,2t,x)$ has the repeated factor $(x^{2}-tx-1)^{2}$.  The cofactor is a quartic which defines the same extension of $\kk$ as $P_{t}(x)$.
\end{rems}
%
%
\section{ $T(m,A,x)$ and number field extensions}\label{s:tmaxsimpnfext}
We now apply the preceding results when $\kk$ is a number field and \\$m,A\in\mathcal{O}_{\kk}$, when the zeroes of $T(m,A,x)$ are units.  We continue to assume that $u$ is defined via Eq.~\eqref{e:srat2}.

If $m^{2}-4\notin\kk^{2}$ and $w^{2}y^{2}\ne 0$, the zeroes of $T(m,A,x)$ are all simple by Proposition \ref{p:simple}.  We have the following result:
\begin{thm}[Real and complex zeroes]\label{t:signature}Let $\kk$ be a number field, $\kk\subset\R$, $m,A\in\mathcal{O}_{\kk}$, $m^{2}-4\notin\kk^{2}$.  Then the number of real zeroes of $T(m,A,x)$ is
\begin{enumerate}\renewcommand{\labelenumi}{(\alph{enumi})}
\item None, if $s^{2}<0$ or $w^{2}<0$;
\item Four, if $s^{2}>0$, $w^{2}> 0$, and $y^{2}<0$;
\item Eight, if $m<-2$;
\item Eight , if $m>2$, $w^{2}>0$, $y^{2}>0$, and
$$\abs{(A+4)^{2}+mA(A+4)+A^{2}}>16;\text{ and}$$
\item None, if $m>2$, $w^{2}>0$, $y^{2}>0$, and
$$\abs{(A+4)^{2}+mA(A+4)+A^{2}}<16,$$
i.e. if $2<m<6$ and $-m-2+2\sqrt{m-2}<A<-2\sqrt{m-2}.$
\end{enumerate}
\end{thm}
\begin{proof}
For (a), $\kk(s,w)\subset\kk(r)$ by Lemma \ref{l:contsw}, so $\kk(r)\nsubseteq\R$ for each zero $r$.

For (b), $\kk(s,w)\subset\R$, so $q_{i}(x)\in\R[x]$.  Now, $\text{disc}(q_{3}(x))$ has the same sign as $\text{disc}(q_{1}(x))$ by Lemma \ref{l:identqs}(b), while $\text{disc}(q_{2}(x))$ and $\text{disc}(q_{4}(x))$ have the opposite sign by Eqs.\ \eqref{e:dp12} and \eqref{e:dp14}.

For (c), if $m<-2$ then $E=\kk(s,w,y)\subset\R$, so the $\text{disc}(q_{i}(x))$ all have the same sign.  Here $u<0$, so by Lemma \ref{l:genrecip2}(d), $\text{disc}(q_{1}(x))>0$.

If $m>2$, $w^{2}>0$, and $y^{2}>0$, then again the $\text{disc}(q_{i}(x))$ all have the same sign, but here $u>0$.  Again using Lemma~\ref{l:genrecip2}(d), $\text{disc}(q_{1}(x))$ and $\text{disc}(q_{2}(x))$ are both $>0$ or both $<0$, according as whether  $\abs{\alpha\beta}>4u$ or $\abs{\alpha\beta}<4u$. By Lemma~\ref{l:genrecip2}(a), $\alpha\beta=A+4+Au$.  Multiplying by the corresponding conditions for $q_{3}(x)$ and $q_{4}(x)$ gives (d) and (e). 
\end{proof}
\begin{rems}
Suppose $m,A\in\R$.  If $m<2$ (in particular, when $s^{2}<0$), then $w^{2}>0$ and $y^{2}>0$. Since Max$(\abs{A},\abs{-m-2-A})\ge\abs{(m+2)/2}$, we have Max$(w^{2},y^{2})\ge (m-6)^{2}/4\ge 0$, so $w^{2}$ and $y^{2}$ cannot both be negative.

The cases $w^{2}<0$ and $y^{2}<0$ correspond to related octics.  Thus, if $m>2$ and $A$ is such that $w^{2}y^{2}<0$, one of the two related octics has signature $(0,4)$ and the other has signature $(4,2)$.

Suppose $T(m,A,x)$ is irreducible with $G\cong$ $_{8}T_{11}$.  Then the fixed field of complex conjugation is non-normal over $\kk$ in cases (a) and (b).  In case (e), the fixed field of complex conjugation is the elementary Abelian extension $E/\kk$.
\end{rems}
We use the Washington's cyclic quartic case $T(t^{2}+2,-t^{2}-2t-4,x)=P_{t}^{2}(x)$ of \S\ref{ss:washquart} to deal with repeated zeroes.
\begin{thm}\label{t:sig2}
Let $\kk\subset\R$, $t\in\mathcal{O}_{\kk}$, and $T(t^{2}+2,-t^{2}-2t-4,x)=P_{t}^{2}(x)$ as in \textup{\S\ref{ss:washquart}}.  Then $P_{t}(x)$ has $4$ real zeroes if $\abs{t+1}>1$, and none if $\abs{t+1}<1$.
\end{thm}
\begin{proof}
In \S\ref{ss:washquart}, the choice $A=-t^{2}-2t-4$ makes $\text{disc}(q_{1}(x))=t(t+2)(u+1)$ in Eq.~\eqref{e:dwz}.  Also in \S\ref{ss:washquart}, $u=v^{2}$ where $v^{2}-tv-1=0$.  Thus, $u+1\ge 1$, so $\text{disc}(q_{1}(x))$ has the same sign as $t(t+2)=(t+1)^{2}-1$.
\end{proof}
The following result depends only on the fact that the zeroes of $T(m,A,x)$ are units when $m$ and $A$ are algebraic integers.  We let $\sim$ indicate associates.
\begin{prop}[Exceptional sequences]\label{p:unitassoc}
Let $\kk$ be a number field, \\$m,A\in\mathcal{O}_{\kk}$, $m^{2}-4\notin\kk^{2}$, and $T(m,A,r)=0$.  Then
\begin{enumerate}\renewcommand{\labelenumi}{(\alph{enumi})}
\item $r +1\sim r+u$  in $\kk(r)$.
\item If $m-2\in\mathcal{O}_{\kk}^{\times}$, then $r, r+1, r+ u$ is an exceptional sequence of three units.
\end{enumerate}
\end{prop}
\begin{proof}
For part (a), $\sigma(r)=(-r-1)/(r+u)$ is a zero of $T(m,A,x)$, hence is a unit.  For part (b), use $P_{s}(-1)=\bar{P}_{s}(-1)=m-2$ and $(u-1)^{2}=( m-2)u$.
\end{proof}
\begin{rem}
The three units in (b) may be incorporated in various normalized \emph{four}-term exceptional sequences  as in \cite{lenstra:euclidean}.  One such is  $0$, $1$, $u$, $-r$.
\end{rem}
We obtain additional units and associates when $m$, $A$, and $w\in\mathcal{O}_{\kk}$.    The following result does \emph{not} require that $[E\!:\!\kk]=4$.
\begin{prop}[``Constellations'' of units and associates]\label{p:constell}
Let $\kk$ be a number field, $m,d\in\mathcal{O}_{\kk}$, $d\mid m-2$, and $A$, $w$ and $P_{w}(x)$ as in \textup{Eqs.~\eqref{e:dihedparms}} and \textup{\eqref{e:dihedquartic}}.  If $P_{w}(r)=0$, then
\begin{enumerate}\renewcommand{\labelenumi}{(\alph{enumi})}
\item $r+1\sim r+u\sim r+1+d$, and $r \sim r+1+d+u\sim 1$.
\item If $d\in\mathcal{O}_{\kk}^{\times}$, then all the quantities in \textup{(a)} are units.
\end{enumerate}
\end{prop}
\begin{proof}
In this case, $q_{1}(x)=x(x+1+d) +u(x+1)=x(x+1+d+u)+u$, giving (a).  For (b), note that in Eq. \eqref{e:dihedquartic}, $P_{w}(-1)=d^{2}$.  Thus, if $d\in\mathcal{O}_{\kk}^{\times}$, $r+1\sim 1$, and the result follows.
\end{proof}
%
%
\section{Extensions of $\kk=\Q$}\label{s:newsimplest}
We let $\kk=\Q$ and $m,A\in\Z$.  For $\abs{m},\abs{A+(m+2)/2}\le 10^{4}$, a simple numerical sweep found that $[E\!:\!\Q]=8$ for over $95\%$ of pairs $(m,A)$.  So it appears that for $m,A\in\Z$, $T(m,A,x)$ ``typically'' produces non-normal octic fields whose normal closures over $\Q$ have Galois group $G\cong$ $ _{8}T_{11}$.
%
\subsection{Real and complex fields}\label{ss:realcomplex}
The only $(m,A)\in\Z\times\Z$  to which Theorem \ref{t:signature}(e) applies, is $(4,-3)$.  Here, $A=-(m+2)/2$, the ``degenerate'' case $d=-2$ of  Eq.~\eqref{e:dihedquartic}; we have $G\cong V_{4}$ and $L=\Q(\zeta_{12})$.  The only $t\in\Z$ for which $\abs{t+1}<1$ as in Theorem \ref{t:sig2} is $t=-1$.  In this case, $L=\Q(\zeta_{5})$.

Apart from these cases, assuming $m\ne\pm 2$, by Theorem \ref{t:signature} the splitting field of  $T(m,A,x)$ is totally real for $m,A\in\Z$ unless $s^{2}<0$ ($m=-1$, $0$, or $1$); or, $m>2$ and $w^{2}<0$ or $y^{2}<0$.  Now $y^{2}<0$ requires that $A^{2}<4(m-2)$, so non-real fields are relatively rare.  When $[E\!:\!\Q]<8$, they are even less common.  We have the following result:
\begin{thm}\label{t:mostlyreal}
Let $m,A\in\Z$, $s^{2}w^{2}y^{2}\ne 0$, $L$ the splitting field of $T(m,A,x)$.
\begin{enumerate}\renewcommand{\labelenumi}{(\alph{enumi})}
\item If $yw\in\Z$, then $L$ is totally real unless $(m,A)=(1,-4)$, $(1,1)$, or $(4,-3)$.
\item If $swy\in\Z$, then $L$ is totally real.
\item If $m\in\Z$, $d\in\Z$, $d\mid m-2$, $d^{2}\le\abs{m-2}$, and $A$ and $P_{w}(x)$ are as in \textup{Eqs.~\eqref{e:dihedparms}} and \textup{\eqref{e:dihedquartic}}, then $L$ is totally real unless $(m,d)=(-1,\pm1)$, $(0,\pm1)$, $(1,\pm 1)$, $(4,-1)$, or $(m,-1)$ with $m> 4$.
\item If $sw\in\Z$, then $L$ is totally real unless $(m,A)=(7,-4)$ or $(7,1)$.
\end{enumerate}
\end{thm}
\begin{proof}
For (a) and (b), one simply checks the criteria of Theorem \ref{t:signature}, keeping in mind that $w^{2}$ and $y^{2}$ cannot both be negative.  For (c), we have $w^{2}>0$, and if $4\le d^{2}\le\abs{m-2}$, $\abs{(m-2)/d+d}\le\abs{(m+2)/2}$, so $y^{2}\ge (m-6)^{2}/4$.

For (d), we have the ``polynomial square root'' identity
$$s^{2}w^{2} = (m^{2} + Am + 2A + 4)^{2}-4A(A+4)m -8(A^{2} +4A +8).$$
If $s^{2}w^{2}$ is a perfect square, $s^{2}w^{2}=(m^{2} + Am + 2A + 4-2k)^{2}$ for some $k\in\Z$.  Now here, $L$ is totally real unless $m>2$ and $A^{2}<4(m-2)$.  This makes the ``remainder'' $-4A(A+4)m -8(A^{2} +4A +8)$ so small, we only need to consider the cases $-4<A<0$  and $0\le k \le 4$.  The only ones where $sw\in\Z$ and $s^{2}w^{2}y^{2}\ne 0$ have $k=1$, namely $(m,A)=(7,-4)$ and $(7,1)$.
\end{proof}
The only cases where $L/\Q$ is Abelian and non-real are $(m,A)=(3,-3)$ or $(3,-2)$ $(L=\Q(\zeta_{5}))$; $(4,-3)$ $(L=\Q(\zeta_{12}))$; $(1,-4)$ or $(1,1)$ $(L=\Q(\zeta_{15}))$; $(7,-4)$ or $(7,-5)$ $(L=\Q(\zeta_{20}))$; and $(7,1)$ or $(7,-10)$ ($L\subset\Q(\zeta_{95}),[L\!:\!\Q]=8$).  The pairs $(m,A)=(7,-4)$ and $(7,1)$ give ``Murphy's twins.'' In both cases, the splitting field of one is $\mathbb{Q}(\zeta_{5})$, while that of the other is the real subfield of $L$.  The $C_{4}\times C_{2}$ cyclotomic field $\Q(\zeta_{16})$ is ``left out,'' although its real subfield is the splitting field of $P_{t}(x)$ for $t=2$.

All quaternion fields defined by $T(m,A,x)$ for $m,A\in\Z$ are totally real, by Theorem \ref{t:mostlyreal}(b).
%
\subsection{Special units}\label{ss:unitsassocs}
We have $\Z^{\times}=\{-1,1\}$, so Propositions~\ref{p:unitassoc}(b) and \ref{p:constell}(b) only apply with $m\in\{1,3\}$ and $d\in\{-1,1\}$, respectively.

The 1-parameter family $T(1,A,x),A\in\Z$, produces fields which are totally imaginary by Theorem \ref{t:signature}(a), and with exceptional sequences of three units by Proposition~\ref{p:unitassoc}(b).  We have $[E\!:\!\Q]=8$ except for $A\in\{-3,0\}$ or $\{-4,1\}$, when $[E\!:\!\Q]=4$.  Eq.\ \eqref{e:dihedparms} gives $A=-3$ when $d=-1$; the quartic factors $P_{w}(x)=x^{4} + x^{3} + 2x^{2} + 2x + 1$ and $\bar{P}_{w}(x)=x^{4} + 5x^{3} + 8x^{2} + 4x + 1$ of $T(1,-3,x)$, both have the minimum quartic discriminant (See \cite{hua:intronumthr}) of 117.   The related octic, $T(1,0,x)=x^{8} - 3x^{6} + 3x^{5} + 14x^{4} + 15x^{3} + 9x^{2} + 3x + 1$, is  a defining polynomial for the splitting field, the Hilbert Class Field of $\Q(\sqrt{-39})$.

Both related octics $T(1,-4,x)$ and $T(1,1,x)$ are irreducible in $\Z[x]$, with $G\cong C_{4}\times C_{2}$, by Proposition \ref{p:c4c2andquatrn}(a) with $t=1$.  The splitting field is $\Q(\zeta_{15})$.

The 1-parameter family $T(3,A,x)$, $A\in\Z$, also produces fields with exceptional sequences of three units.  But these fields are totally real, except when $A\in \{-6,1\}$, $\{-5,0\}$, $\{-4,-1\}$, or $\{-3,-2\}$. Since $4(m-2)=4$ for $m=3$, $w^{2}\notin\Z^{2}$ and $y^{2}\notin\Z^{2}$ unless $w^{2}y^{2}=0$.  It is also not hard to check that with $m=3$ and $A\in\Z$, $w^{2}y^{2}\notin\Z^{2}$ and $s^{2}w^{2}y^{2}=5w^{2}y^{2}\notin\Z^{2}$ unless $w^{2}y^{2}=0$. Thus, $[E\!:\!\Q]=8$ unless either $5w^{2}\in\Z^{2}$ or $5y^{2}\in\Z^{2}$.  Taking $5w^{2}\in\Z^{2}$ gives a family of Murphy's twins  (See Eqs.\ \eqref{e:sw1} - \eqref{e:barsw2}).  These define a family of cyclic quartic fields which (apart from $\mathbb{Q}(\zeta_{5})$) are real, contain $\Q(\sqrt{5})$, and have exceptional sequences of three units by Proposition \ref{p:unitassoc}(b).

Using Pari-GP, we checked $T(m,A,x)$ for small values of $m$ and $A$ against the octic fields with small discriminants listed in the tables of \cite{cohenetal:octics}. We found that $T(3,-6,x)$, $T(1,-2,x)$, and $T(1,-14,x)$ all define the two conjugate octic fields of minimum discriminant for signature $(0,4)$ and $G\cong _{8}T_{11}$.  Each of these polynomials gives exceptional sequences of $3$ units in those fields.

If $m,d\in\Z$, $m\ne\pm 2$, $d\mid m-2$, and $P_{w}(r)=0$, $P_{w}(x)$ as in Eqs.~\eqref{e:dihedparms} and \eqref{e:dihedquartic}, then $P_{w}(\sigma(r))\ne 0$, but $\sigma(r)=(-r-1)/(r+u)$ is a unit in $\Q(r)$.  When $d\ne -1$, this allows us to exhibit a system of three independent units  for the field $\Q(r)$ when this is a totally real quartic field.  The values $d=-1$ and $d=1$ give the 1-parameter families
\begin{subequations}\label{e:specdihedquarts}
\begin{align}
\label{e:sda}&P_{w}(x)=x^{4}+mx^{3}+(m+1)x^{2}+2x+1,\text{ and}\\
\label{e:sdb}&P_{w}(x)=x^{4}+(m+4)x^{3}+(3m+5)x^{2}+(2m+2)x+1,
\end{align}
\end{subequations}
respectively, to which Proposition~\ref{p:constell}(b) applies.  If $P_{w}(r)=0$ in Eq.~\eqref{e:sdb}, then (with $u$ defined via Eq.~\eqref{e:srat2})  $r$, $r+1$, $r+2$, $r+u$, and $r+2+u$ are all units.
%
\subsection{Murphy's twins}\label{ss:murphtwins}
If $m,A\in\Z$ and $sw\in\Z$, the Murphy's twins $P_{sw}(x)$ and $\bar{P}_{sw}(x)$ define cyclic quartic number fields when the conditions of Theorem \ref{t:irrcyc} hold, and (with $u$ and $\sigma$ defined via Eq.~\eqref{e:srat2}) their zeroes are units satisfying \eqref{e:M} with $n=4$.  We give a description of Murphy's twins cyclic quartic fields having a given quadratic subfield, based on standard results about norms from real quadratic fields.  We show that any quadratic field which is contained in \emph{some} cyclic quartic field, is a subfield of infinitely many Murphy's twins cyclic quartic fields.

Let $d>1$ be a squarefree integer, $K=\Q(\sqrt{d})$, and assume that $m\in\Z$, $m^{2}-4=dN^{2}$, $N\in\Z$.  Then $(m+N\sqrt{d})/2$  is a unit of norm $1$ in $\mathcal{O}_{K}$, so $m$ and $N$ are (to within sign) generalized Lucas and Fibonacci numbers, respectively, which we describe as follows.

Let $\varepsilon>1$ be the fundamental unit of $K$, and $\bar{\varepsilon}$ its conjugate.  Let
\begin{equation}\label{e:lf}\tag{LF}
\varepsilon^{n}=\frac{L_{n}+F_{n}(\varepsilon-\bar{\varepsilon})}{2}, \text{ where } L_{n},F_{n}\in\Z.
\end{equation}
The properties in \cite{hardywright:introthrnum}, Theorem 179, where $d=5$, and $L_{n}$ and $F_{n}$ are the Lucas and Fibonacci numbers, easily generalize to any squarefree $d>1$.  We have $\Q(\sqrt{m^{2}-4})=\Q(\sqrt{d})$ when $m=\pm L_{n}$, where $n$ is arbitrary if $\mathcal{N}(\varepsilon)=+1$, but $n$ must be \emph{even} if $\mathcal{N}(\varepsilon)=-1$.

Now we want $w^{2}=dY^{2}$ for some $Y\in\Z$.  Taking $X=m+A+2\in\Z$, we may express this as $(X+Y\sqrt{d})/2=\pi\in\mathcal{O}_{K}$, $\mathcal{N}(\pi)=m-2$. Then $X^{2}+(m-2)^{2}=d(N^{2}+Y^{2})$, so $d$ is the sum of two squares, as expected if $K$ is contained in a cyclic quartic field.  We have the following result:
\begin{prop}\label{p:twins}
Let $d>1$ be squarefree, $d=a^{2}+b^{2}$, $K=\Q(\sqrt{d})$, $L_{n}$ and $F_{n}$ as in \textup{Eq.\ \eqref{e:lf}}.  There are infinitely many $m$ such that $\Q(s)=K$, for each of which there are infinitely many $A\in\Z$ such that $T(m,A,x)=P_{sw}(x)\bar{P}_{sw}(x)$ in $\Z[x]$.
\end{prop}
\begin{proof}
If $\mathcal{N}(\varepsilon)=-1$, then  for any $j\in\Z$, we obtain $sw\in\Z$ by taking
$$m=L_{2k},\; X+Y\sqrt{d}=2(\varepsilon^{2k} -1)\varepsilon^{2j-1},\;\text{and}\;A=-m-2\pm X.$$
If $\mathcal{N}(\varepsilon)=+1$, choose $n$ and the $\pm$ sign so that $u=\pm\varepsilon^{n}\equiv 1\pmod{a\mathcal{O}_{K}}$.  Then $(u-1)/a\in\mathcal{O}_{K}$, so $sw\in\Z$ for any $j\in\Z$, if
$$m=\pm L_{n}, X+Y\sqrt{d}=2((u-1)/a)(b-\sqrt{d})\varepsilon^{j},\text{ and}\;A=-m-2\pm X.$$
\end{proof}
\begin{rem}
The equation $\mathcal{N}(\pi)=m-2$ may well have other solutions than those given in Proposition \ref{p:twins}.
\end{rem}
Thus, there are Murphy's twins cyclic quartic fields containing $\Q(\sqrt{d})$ whenever $d>1$ is squarefree and the sum of two squares.    But suppose $d$ is even as well; and further, that $K=\Q(\sqrt{d})$ has $\mathcal{N}(\varepsilon)=+1$.  Then, $K$ is \emph{not} a subfield of any of Washington's cyclic quartic fields (whose quadratic subfields always have $\mathcal{N}(\varepsilon)=-1$).  It is also easy to show that $K$ is not a subfield of any ``simplest'' quartic field. The smallest such $d$ is $34$.

The only non-real Murphy's twins cyclic quartic field is $\mathbb{Q}(\zeta_{5})$, which occurs when $(m,A)=(3,-3)$, $(3,-2)$, $(7,-4)$, and $(7,1)$.

We can make  $P_{(sw)'}(x)=x^{4}P_{sw}(1/x)$ as per Proposition \ref{p:diffgen}, with parameter $m'=m$, the same square root $s$ of $m^{2}-4$, and $(sw)'=s\cdot w'$, specifying $A'$ and $w'$ by
\begin{subequations}
\begin{equation}\label{e:conjshift}\frac{m+2+A'+w'}{2}=\left(\frac{m+2+A-w}{2}\right)\left(\frac{m+s}{2}\right),\text{ that is}\end{equation}
\begin{equation}A'=\frac{m^{2}+mA-4-sw}{2}\text{ and }w'=\frac{(m+A+2)s-mw}{2}.\end{equation}
\end{subequations}
Of course, $(m,A',-s,-w')$ gives the same $P_{(sw)'}(x)$ as $(m,A',s,w')$.  Note, however,  that if $sw\ne 0$, $P_{sw}(x)$ and $\bar{P}_{sw}(x)$ produce different values of $A'$.  For the $P_{t}(x)$ in \S\ref{ss:washquart}, $A'=-t^3 - t^2 - 2t - 4$ and $(sw)'=-t^5 - 4t^3$ give $P_{(sw)'}(x)=x^{4}P_{t}(1/x)$.

The case $m=3$ is particularly simple.  Here, $d=5$, $\varepsilon=\tau$, the ``golden ratio,'' and $(X+Y\sqrt{5})/2$ has norm $m-2=1$.  With $X=L_{2j}$, $A=-5+L_{2j}$ and $sw=5F_{2j}$ (the original Lucas and Fibonacci numbers), we obtain
\begin{subequations}  \label{e:twins3}
\begin{align}
\label{e:sw1}&P_{sw}(x) = x^{4} + (5 - L_{2j} )x^{3} + (9 - 5F_{2j-1})x^{2} + (5 - L_{2j-2})x + 1\text{ and}\\
\label{e:barsw1}&\bar{P}_{sw}(x) = x^{4} + (5 - L_{2j} )x^{3} + (9 - 5F_{2j+1} )x^{2} + (5 - L_{2j+2} )x + 1.
\end{align}
\end{subequations}

With  $A=-5-L_{2j}$ and $sw=5F_{2j}$, we obtain
\begin{subequations}
\begin{align}
\label{e:sw2}&P_{sw}(x) =x^{4}+ (5 + L_{2j} )x^{3} + (9 + 5F_{2j+1})x^{2} + (5 + L_{2j+2})x + 1 \text{ and}\\
\label{e:barsw2}&\bar{P}_{sw}(x) = x^{4} + (5 + L_{2j})x^{3} + (9 + 5F_{2j-1})x^{2} + (5 + L_{2j-2} )x + 1.
\end{align}
\end{subequations}

As suggested by Eq.~\eqref{e:conjshift}, in both families $x^{4}P_{sw}(1/x)$ is obtained by taking $\bar{P}_{sw}(x)$ and shifting the index $j$.

If $P(r)=0$ for one of the cyclic quartics in Eqs.\ \eqref{e:sw1}-\eqref{e:barsw2} and $u$ is defined via Eq. \eqref{e:srat2}, $r$, $r+1$, $r+u$ is an exceptional sequence of $3$ units by  Proposition \ref{p:unitassoc}(b).  If $T(3,\pm L_{2j},r)=0$, then $r$, $r+1$, $r+u$ is also an exceptional sequence of $3$ units, which are generally of degree $8$ over $\Q$.

It is easy to show that for a given $m\in\Z$, there are only finitely many $A\in\Z$ for which either $y^{2}\in\Z^2$ or $y^{2}w^{2}\in\Z^2$.  Consequently, for a given $m$ there can be only finitely many $A$ yielding Murphy's twins which do \emph{not} have distinct splitting fields.  The octics with $(m,A)=(-3,-4)$, $(-7,8)$, and $(-66,13)$, are the only examples we know where  $s^{2}w^{2}y^{2}\ne 0$ and the ``twins'' define the \emph{same} $C_{4}$ extension of $\Q$.
%
\subsection{The cases $[L\!:\!\Q]=8$ when $wy\in\Z$, $swy\in\Z$, and $w\in\Z$}\label{ss:deg8recap}
We apply Proposition \ref{p:c4c2andquatrn} to the cases $wy\in\Z$ and $swy\in\Z$ by taking $t\in\Z^{+}$.  In (a), we have $[E\!:\!\Q]=4$ when $t\ne 2$; and in (b), $[E\!:\!\Q]=4$ when $t>1$.  Clearly $t^{2}-4$ and $t^{2}+4$ are both squarefree for a positive proportion of $t\in\Z^{+}$; likewise for $t^{2}+4$ and $t^{2}+8$, so these families of $C_{4}\times C_{2}$ and $Q_{8}$ number fields are infinite.  The family of  quaternion fields mentioned in the Remarks after Proposition \ref{p:c4c2andquatrn} may also be shown to be infinite.

When $m,d\in\Z$, $m\ne\pm 2$, and $d\mid m-2$ in Eqs.~\eqref{e:dihedparms} and \eqref{e:dihedquartic}, it is not hard to show that $P_{w}(x)$ is irreducible in $\Z[x]$ unless $d=-2$ and $m=6$.  Apart from the special cases described after Eqs.~\eqref{e:dihedparms} and \eqref{e:dihedquartic}, the only instances we know where $P_{w}(x)\in\Z[x]$ has $G\ncong D_{4}$ are $(m,A)=(-3,5)$, $(-7,-3)$ and $(-66,51)$.  In these instances, $s^{2}y^{2}\in\Z^{2}$, and $G\cong C_{4}$.  The related octics are the examples mentioned at the end of \S\ref{ss:murphtwins}.
%
\subsection{Regulator formulas}\label{ss:regs}
Let $\kk=\Q$, and $m,A\in\Z$.  If $T(m,A,r)=0$ and $F=\Q(r)$, then $\mathcal{O}_{F}^{\times}$ has rank 3 if $F$ has signature $(0,4)$ or $(4,0)$. We show that in most such cases, the system  $\langle\zeta,\varepsilon,r,\sigma(r)\rangle$ has rank $3$, where $\langle\zeta\rangle$ is the torsion units in $F$, $\varepsilon$ is the fundamental unit of a real quadratic subfield of $F$, $\sigma(r)$ is as in Eq.~\eqref{e:C4} with $u$ as per Eq.~\eqref{e:srat2}.  The regulator formulas are similar to those in \cite{lazarus:classunitquart} and \cite{wash:quartics}.

In the following three results (which we state without proof), by Theorem \ref{t:irrcyc} we can treat $\sigma$ as a field isomorphism.  Checking that the $\ln^{2}()$ values are not both $0$, is left as an exercise for the reader.
\begin{prop}\label{p:imagoctsnegs}
Let $m\in\{-1,0,1\}$, $(m,A)\ne(-1,-3)$, $(-1, 1)$, $(0,-3)$, $(0,-1)$, or $(1,-3)$.  Let $\varepsilon>1$ be the fundamental unit of $\Q(w)$ and $\langle\zeta\rangle$ the torsion units in $\Q(r)$.  Then
$$\textup{Reg}\langle\zeta,\varepsilon,r,\sigma(r)\rangle=16\ln(\varepsilon)\left(\ln^{2}\abs{r} + \ln^{2}\abs{\sigma(r)}\right) \ne 0.$$
\end{prop}
\begin{prop}\label{p:imagoctsposs}
Let $m>2$ and $(m+2+A)^{2}<4(m-2)$.  Let $\varepsilon>1$ be the fundamental unit of $\Q(s)$ and $\langle\zeta\rangle$ the torsion units in $\Q(r)$.  Then
$$\textup{Reg}\langle\zeta,\varepsilon,r,\sigma(r)\rangle=4\ln(\varepsilon)\left(\ln^{2}\abs{r^{2}/u} + \ln^{2}\abs{\sigma(r)^{2}u}\right) \ne 0.$$
\end{prop}
\begin{prop}\label{p:realcycquart}
Let $\abs{m}>2$.  Assume $T(m,A,x)=P_{sw}(x)\bar{P}_{sw}(x)$ in $\Z[x]$, $(m,A)\ne (3,-2)$ or $(3,-3)$.  Let $\varepsilon>1$ be the fundamental unit of $\Q(s)$.  If $T(m,A,r)=0$, $[\Q(r)\!:\!\Q]=4$, and $\Q(r)$ is real, then
$$\textup{Reg}\langle-1,\varepsilon,r,\sigma(r)\rangle=\frac{1}{2}\ln(\varepsilon)\left(\ln^{2}\abs{r^{2}/u} + \ln^{2}\abs{\sigma(r)^{2}u}\right) \ne 0.$$
\end{prop}
For $P_{w}(x)$ as in Eqs.~\eqref{e:dihedparms} and (b), $\sigma(r)$  is a unit in $\Q(r)$, but it is \emph{not} an algebraic conjugate of $r$.  In this case we have the following result:
\begin{prop}\label{p:noncycrealquart}Let $m,d\in\Z$, $d\mid m-2$, $d\ne -1$, $d^{2}\le\abs{m-2}$.  Let $P_{w}(x)$ be as in \textup{Eqs.~\eqref{e:dihedparms} and (b)}. Assume $P_{w}(x)$ is irreducible with signature $(4,0)$, and $\varepsilon>1$ is the fundamental unit of $\Q(s)$. Then up to a factor of  $(1+\textit{\large{O}}\left(\abs{d/m-\!2}\right))=(1\!+\!\textit{\large{O}}(1/\sqrt{\abs{m-\!2}}))$,
\begin{align*}
&\textup{Reg}\langle -1,\varepsilon,r,\sigma(r)\rangle\approx\\
&2\ln(\varepsilon)\abs{\ln^{2}\abs{m-2} + \ln\abs{(d+\!1)^{2}/d}\ln\abs{m-2} - \ln\abs{d}\ln\abs{d+1}}.\\
\end{align*}
\end{prop}
\begin{proof}
With $R=1/(m-2)$, we can express the zeroes of $x^{2}-mx+1$ or $x^{2}-(1/R+2)x+1$ as formal power series in $R$, $u_{1}=1/R + 2 -R +2R^2 +\ldots$ and $u_{2}=R-2R^{2}+\ldots$, which converge for $\abs{m-2}>4$.

Taking $q_{1}(x)=x^{2}+(1+d+u_{1})x+u_{1}$, with a zero $r_{1}\approx -1$, we obtain a formal series for $r_{1}$ with terms $p_{k}(d)R^k$, where $p_{k}(d)\in\Z[d]$ has degree $k$. This gives a series for  $\sigma(r_{1})=-(r_{1}+1)/(r_{1}+u_{1})$.  Now $P_{w}(x) = q_{1}(x)q_{4}(x)$, where $q_{4}(x)=x^{2}+(d+1+u_{2})x+u_{2}$, which has a zero $r_{3}\approx -d-1$.  (The coefficient of $R^{k}$ in the series for $r_{3}$ has a power of $d+1$ in the denominator.)  Using $\varepsilon$, $r_{1}$, $\sigma(r_{1})$; $\varepsilon'=\pm\varepsilon^{-1}$, $r_{3}$, $\sigma(r_{3})=-(r_{3}+1)/(r_{3}+u_{2})$; and $\varepsilon$, $r_{2}=u_{1}/r_{1}$, $\sigma(r_{2})=u_{2}/\sigma(r_{1})$ to form the regulator determinant then gives the result.
\end{proof}
\begin{rems}When $d=-1$ the quartic fields defined by the $P_{w}(x)$ in Eq.~\eqref{e:sda} are totally real when $m<-2$, but  $r+1,r+u\in\langle -1,\varepsilon,r\rangle$, so $\langle -1,\varepsilon,r,\sigma(r)\rangle$ has rank $2$ at most.  When $m=4-t-t^{2}$, $t\in\Z$, $t>2$, we have $G=V_{4}$.  In this case,  the fundamental units of the $3$ quadratic subfields of the splitting field give $3$ independent units.
\end{rems}
Now $\sigma(r)=(-r-1)/(r+u)=-(r+1)^{2}/rq_{1}(-1)$.  Thus, $r+1$ and $r+u$ are both units precisely when  $q_{1}(-1)= (m+A+2+w)/2$ is a \emph{unit}.   With $\kk=\Q$ and $m,A\in\Z$, we then find that $q_{1}(-1)\in\langle\zeta,\varepsilon\rangle$, so
\begin{equation}\label{e:unitindex}
[\langle\zeta,\varepsilon,r,r+1\rangle\!:\!\langle\zeta,\varepsilon,r,\sigma(r)\rangle]=2, \text{ when }q_{1}(-1)=\pm 1,\text{ or }m-2=\pm 1.
\end{equation}
This gives a ``one-half'' regulator in Proposition \ref{p:imagoctsnegs} when $m=1$; in Proposition \ref{p:imagoctsposs} when $m=3$ and $A=-6$, $-5$, or $-4$; for the quartics in Eqs.~\eqref{e:sw1} - \eqref{e:barsw2}; and for the quartics in Eq.~\eqref{e:sdb}.

The regulators in Propositions \ref{p:imagoctsnegs}--\ref{p:noncycrealquart} can be extremely small.  We estimate the ``one-half'' regulator $R_{1}=\text{Reg}\langle\zeta,\varepsilon,r,r+1\rangle$ for infinite families with $\varepsilon=\tau$, the ``golden ratio;'' this makes the factor $\ln(\varepsilon)$ as small as possible.
  
When $m=1$, $A\in\Z$, $A\ne -3$, and $T(1,A,r)=0$, $F=\Q(r)$ is a totally imaginary octic field.  Here $y^{2}=A^{2}+4$.  By Eq.~\eqref{e:dp14}, if $A^{2}+4$ is squarefree, $\Delta(F/\Q)=(\Delta(\Q(s,w)/\Q))^{2}(A^{2}+4)^{2}$.  Taking  $A=L_{2j-1}-3$ makes $\varepsilon=\tau$.  Then, refining $r\approx -1$ gives, assuming $(L_{2j-1}-3)^{2}+4$ is squarefree,
\begin{equation}\label{e:tinyr1meq1}
R_{1}=\frac{1}{2}\ln(\tau)\ln^{2}\left(\frac{1}{15^{4}}\Delta(F/\Q)\right)\left(1+\textit{\large{O}}\left(1/\abs{A}\right)\right).
\end{equation}
Let $F$ be a real cyclic quartic field given by  Eq.~\eqref{e:sw1}.  Using Lemma \ref{l:genrecip2}(d), $\text{disc}(q_{1}(x))\text{disc}(q_{3}(x))=5(F_{2j-1}-2)^{2}$.  Refining $r\approx -1$, and using Proposition \ref{p:realcycquart},  we find that if $F_{2j-1}-2$ is squarefree and prime to $10$, then
\begin{equation}\label{e:tinyr1meq3}
R_{1}= \frac{1}{4}\ln(\tau)\ln^{2}\left(\frac{1}{5^{2}}\Delta(F/\Q)\right)\left(1+\textit{\large{O}}\left(1/j^{2}\right)\right).
\end{equation}
Replacing $F_{2j-1}-2$ with $F_{2j-1}+2$, and assuming this is squarefree and prime to $10$, we again obtain Eq.~\eqref{e:tinyr1meq3} for the real cyclic quartic fields given by  Eq.~\eqref{e:barsw2}.

Using Proposition \ref{p:noncycrealquart} and Eq.~\eqref{e:dp14}, we obtain an estimate asymptotically equal to that in Eq.~\eqref{e:tinyr1meq3}  for the non-normal quartic fields defined by Eq.~\eqref{e:sdb}, taking $m=L_{2j}$, a Lucas number of even index, if $4L_{2j}^{2}+9$ is squarefree.

The regulator estimates in Eqs.~\eqref{e:tinyr1meq1} and \eqref{e:tinyr1meq3} are comparable to the lower bound in \cite{silv:regest}, for totally imaginary octic fields with a real quadratic but not a real quartic subfield, or for real quartic fields with a quadratic subfield.

We do not know that the squarefreeness conditions are satisfied infinitely often, but we do not know any reason to assume otherwise.  The numbers $F_{2j-1}\pm 2$ are always the product of a Fibonacci number and a Lucas number whose indexes differ by 3, but the squarefreeness question for these is also open.

In \cite{wash:quartics}, Washington obtains a lower bound for the regulators of the cyclic quartic fields defined by $f_{t}(x)$ which proves the system of $3$ units he gives is fundamental when $t$, $t+4$, and $t^{2}+4$ are squarefree, except when $t=1$.  The system in Proposition \ref{p:realcycquart} for $P_{t}(x)$ has the same regulator.  In the case $t=1$, the system $\langle-1,\tau,r,r+1\rangle$ for Eqs.~\eqref{e:sw2} and (b) with $j=0$ is fundamental.

\section{Concluding remarks}\label{s:conclusion}
The condition \eqref{e:M} provides a conceptual unification of all families of number fields previously dubbed ``simplest.''   It also enabled us to re-derive both the ``simplest'' quartic fields and Washington's cyclic quartic fields by elementary methods, and to place both families together in a larger context.

For purely algebraic purposes, the substitution $A\leftarrow -(m+2)/2+\mathbf{v}$ used in the proof of Theorem \ref{t:relocts} may be useful.

However, more sophisticated methods than used here would surely be required to construct algebraic maps $\sigma$ satisfying \eqref{e:M} for $n>4$.  More powerful techniques would probably also refine and elaborate some of our results considerably, particularly those in \S\ref{ss:regs}.  Those wanting to to investigate class number or unit index questions might want to read \cite{gras:specunit}, \cite{gras:tableclassunit}, \cite{lazarus:classunitquart}, \cite{louboutin:effcomp}, \cite{hasse:unitclass} and \cite{hasse:classno} first.

We note that \eqref{e:M} and its $\sigma$-conjugates can be viewed as a system of $n$ constraints which may limit the relative sizes of $r,\sigma(r),\ldots, \sigma^{n-1}(r)$.  This might help account for the previously-noted feature of small regulators in ``simplest'' cyclic number fields of low degree.  In this regard, the factorization of the ``circulant'' matrix determinant (see, for example, Lemma 5.26 in \cite{wash:introcyclo}) might be of interest.

We chose the variable names $s$ and $w$ prior to learning about the ``simplest'' quartic fields or Washington's cyclic quartic fields.  The fact that the conditions $s=0$ and $w=0$ turned out to produce the ``simplest'' quartic fields and Washington's cyclic quartic fields respectively, is a coincidence that absolutely delights the author.

While investigating Eq.~\eqref{e:C4}, we noticed that $f:x\mapsto(-x-1)/(x+u)$ has compositional order $10$ when $u^{2}+3u+1=0$.  The compositional powers $f^{(3)}$ and $f^{(7)}$ make \eqref{e:M} a formal identity with $n=10$.  With $K=\Q(\sqrt{5})$,
$$\sum_{k=0}^{9}f^{(k)}(x)=2(A+Bu),\;A,B\in\Z$$
produces a family of totally real $C_{10}$ extensions of $K$ defined by units whose conjugates satisfy \eqref{e:M}, and which have exceptional sequences of \emph{four} units.  The normal closure over $\Q$ has ``generic'' Galois group $_{20}T_{53}$.
\section*{Acknowledgements}
Without Phil Carmody's willingness to perform many Pari-GP calculations, this work would hardly have begun; and his subsequent guidance in using the software was most useful in completing it.  My thesis advisor Leon McCulloh gave indispensable advice and assistance with the submission process, and pointed out simplifications of some of the arguments.  J\"{u}rgen Kl\"{u}ners suggested rational rather than integer parameter values, and using Magma to compute Galois groups of 1-parameter families of polynomials.  Derek Holt's description of $_{8}T_{11}$ as a central product, and Laurent Bartholdi's use of GAP were instrumental in understanding $_{8}T_{11}$ and $_{20}T_{53}$.  H. W. Lenstra, Jr. and Gerhard Niklasch provided historical context and specific references for the terminology and applications of exceptional units.

The members of the North Dakota NMBRTHRY listserv supplied more helpful suggestions and calculations than I can list here.  Claus Fieker provided very helpful guidance about $_{8}T_{11}$.  Franz Lemmermeyer, Attila Peth\H{o}, Siman Wong and Volker Ziegler provided useful bibliographic references.  Duncan Buell, Kok Seng Chua and Odile Lecacheux sent papers. Duncan Buell, George Gras, Patrick Morton, and Lawrence Washington sent papers, and also helped clarify the origins of the term, ``simplest number fields.''  To them, and to many others, my heartfelt thanks.  I hope the ideas and results presented here will in some measure repay their generosity.

\end{document}